\input ./style/arxiv-general.cfg
\documentclass[aap,MSNbibl,dvips]{arximspdf}
\makeatletter
   \@ifpackageloaded{graphicx}{}{\usepackage{graphicx}}
\makeatother
\usepackage{mathrsfs}

\doi{10.1214/15-AAP1118}
\volume{26}
\issue{3}
\pubyear{2016}
\firstpage{1329}
\lastpage{1361}
\docsubty{FLA}

\makeatletter
\def\textsc{}
\def\mid{|}
\newcommand{\eqref}[1]{(\ref{#1})}
\newcommand{\BR}{\mathbb{R}}
\newcommand{\SL}{\sum}
\newcommand{\al}{\alpha}
\newcommand{\be}{\beta}
\newcommand{\ga}{\gamma}
\newcommand{\De}{\Delta}
\newcommand{\La}{\Lambda}
\newcommand{\ME}{\mathbf E}
\newcommand{\CF}{\mathcal F}
\newcommand{\CG}{\mathcal G}
\newcommand{\CE}{\mathcal E}
\newcommand{\MP}{\mathbf P}
\newcommand{\CM}{\mathcal M}
\newcommand{\CH}{\mathcal H}
\newcommand{\CN}{\mathcal N}
\newcommand{\Oa}{\Omega}
\newcommand{\si}{\sigma}
\newcommand{\eps}{\varepsilon}
\newcommand{\Ra}{\Rightarrow}
\newcommand{\ol}{\overline}
\newcommand{\MQ}{\mathbf{Q}}

\newcommand{\taub}{\bolds{\tau}}

\newcommand{\ess}{\operatorname{ess}}

\newtheorem{thmm}{Theorem}[section]

\newtheorem{lemma}[thmm]{Lemma}
\newtheorem{lemmaa}{Lemma}[section]
\newtheorem{prop}[thmm]{Proposition}
\newtheorem{cor}[thmm]{Corollary}

\newproclaim{defn}{Definition}
\newproclaim{asmp}{Assumption}
\newproclaim{notn}{Notation}
\newproclaim{prb}{Problem}

\newproclaim{rmk}{Remark}
\newproclaim{exm}{Example}
\newproclaim{clm}{Claim}
\makeatother

\begin{document}
\begin{frontmatter}

\title{Diverse market models of competing Brownian particles with
splits and mergers\thanksref{T1}}
\runtitle{Splits and mergers}
\thankstext{T1}{Supported in part by NSF Grants
DMS-09-05754, DMS-10-07563, DMS-13-08340 and DMS-14-05210.}

\begin{aug}
\author[A]{\fnms{Ioannis}~\snm{Karatzas}\corref{}\ead[label=e1]{ik@math.columbia.edu}\ead[label=e2]{ik@enhanced.com}}
\and
\author[B]{\fnms{Andrey}~\snm{Sarantsev}\ead[label=e3]{ansa1989@math.washington.edu}}
\runauthor{I. Karatzas and A. Sarantsev}
\affiliation{Columbia University and University of Washington}
\address[A]{Department of Mathematics\\
Columbia University\\
New York, New York 10027\\
USA\\
and\\
INTECH Investment Management LLC\\
One Palmer Square, Suite 441\\
Princeton, New Jersey 08542\\
USA\\
\printead{e1}\\
\phantom{E-mail:\ }\printead*{e2}}

\address[B]{Department of Mathematics\\
University of Washington\\
Box 354350\\
Seattle, Washington 98195\\
USA\\
\printead{e3}}
\end{aug}

%
\received{\smonth{4} \syear{2014}}
%
\revised{\smonth{2} \syear{2015}}

%
\begin{abstract}
We study models of regulatory breakup, in the spirit of Strong and
Fouque [\textit{Ann. Finance} \textbf{7} (2011)
349--374] but with a fluctuating number of companies. An important
class of market models is based on systems of competing Brownian
particles: each company has a capitalization whose logarithm behaves as
a Brownian motion with drift and diffusion coefficients depending on
its current rank. We study such models with a fluctuating number of
companies: If at some moment the share of the total market
capitalization~of a company reaches a fixed level, then the company is
split into two parts of random size. Companies are also allowed to
merge, when an exponential clock rings. We find conditions under which
this system is nonexplosive (i.e., the number of companies remains
finite at all times) and diverse, yet does not admit arbitrage opportunities.
\end{abstract}

%

%
\end{frontmatter}

\section{Introduction}\label{sec1}

Stochastic Portfolio Theory (SPT) is a fairly recently developed area
of mathematical finance. It tries to describe and understand
characteristics of large, real-world equity markets using an
appropriate stochastic framework, and to analyze this framework
mathematically. It was introduced by \textsc{Fernholz} in the late
1990s, and was developed fully in his book \cite{F2002}; a survey of
somewhat more recent developments appeared in \cite{FK2009}.

One feature of real-world markets that this theory tries to account for
is diversity. A market is called \textit{diverse}, if at no time is a
single stock allowed to dominate almost the entire market in terms of
capitalization. To be a bit more precise, let us define the \textit{market
weight} of a certain company as the ratio of its capitalization (stock
price, times the number of shares outstanding) to the total
capitalization of the entire market, across all companies. If no market
weight ever exceeds a certain threshold, a fixed number between zero
and one, then this market model is called \textit{diverse}.

Such models have one very important feature: with a fixed number of
companies and a strictly nondegenerate covariance structure, they
allow arbitrage on certain fixed, finite time-horizons. The market
portfolio can be outperformed in these models using fully invested,
long-only portfolios. This was shown in \cite{F2002}, Chapter~3;
further examples of portfolios outperforming the market are given in
\cite{FK2005,FKK2005}, \cite{FK2009}, Section~11. Some such
models were constructed in \cite{FKK2005,OR2006,Ruf2013,MyOwn1} and \cite{FK2009}, Chapter~9; see also the
related articles \cite{FFR2007,K2008}.

Another feature of large equity markets that SPT tries to capture is
that stocks with larger capitalizations tend to have smaller growth
rates and smaller volatilities. In an attempt to model this phenomenon,
the authors of \cite{BFK2005} introduced a new model of \textit{Competing
Brownian Particles} (CBPs). Imagine a fixed, finite number of particles
moving on the real line; at each time, they are ranked from top to
bottom, and each of them undergoes Brownian motion with drift and
diffusion coefficients depending on its current rank. From these random
motions, one constructs a market model with a finite number of stocks:
the logarithms of the companies' capitalizations evolve as a system of
CBPs. Recently, these systems were studied extensively (see \cite
{PP2008,CP2010,IchibaThesis,Ichiba11,FIK2013,IK2010,IKS2013,IPS2012,PS2010,MyOwn3}) and were generalized in several directions:
\cite{S2011,FIK2013b,KPS2012,Ichiba11,MyOwn5,MyOwn6}. However, these market models are \textit{not}
diverse; see \cite{BFK2005}, Section~7 and Remark~\ref{NoDiv} below.

We would like to alter the CBP-type model a bit, in order to make it
diverse. In real equity markets, diversity is in large part a
consequence of anti-monopolistic legislation and regulation: when a
company becomes dominant, a governmental agency (the ``regulator'')
forcibly splits it into smaller companies. We implement this idea in
our model.

In this paper, we construct a diverse model from the above CBP-based
one. We fix a certain \textit{threshold} between 0 and 1. When a company's
market weight reaches this threshold, the regulatory agency enforces a
breakup of the company into two (random) parts. We also allow for the
opposite phenomenon: companies can merge at random times.

The mechanism for merging companies is as follows: immediately after a
split or merger, we set an exponential clock whose rate depends %
on the number of extant companies. If the clock rings before any market
weight has hit the threshold, the regulatory agency picks two companies
at random as candidates for a possible merger, according to a certain
rule described right below. If the planned action results in a company
with market weight exceeding the threshold, then this putative merger
is suppressed; otherwise, it is allowed to proceed.

We use the following rule for mergers: The company which currently
occupies the highest capitalization rank is excluded from
consideration, and two of the remaining $ N-1 $ companies are chosen
randomly, according to the uniform distribution over the $
{N-1\choose2}$ possible choices. With this rule, and with a threshold
sufficiently close to 1, the merger will always be allowed to proceed.
In this manner, the process of capitalizations evolves as an
exponentiated system of competing Brownian particles, until either (i)
one of the market weights hits the threshold, or (ii) the exponential
clock rings. In case (i), the number of companies will increase by one;
in case (ii), it will decrease by one.

We refer the reader to the very interesting paper \cite{FS2011}, which
considers general (i.e., not just CBP-based) equity market models of
regulatory breakup with a split when a market weight reaches the given
threshold. Mergers in that paper obey a different rule than they do
here: at the moment of any split, there is a simultaneous merger of the
two smallest companies, so the total number of companies remains
constant. We feel that this feature is a bit too restrictive, so in the
model developed here mergers are allowed to happen independently of
splits. This comes at a price, which is both ``technical'' and
substantive: the number of companies in the model is now fluctuating
randomly, in ways that need to be understood before any reasonable
analysis can go through. The foundational theory for generic market
models with a randomly varying number of stocks was developed by
\textsc
{Strong} in the important and very useful article \cite{Strong2011}.

It is also important to stress that the mechanisms enforcing diversity
in the model studied here are quite different from those used in \cite
{FKK2005} or \cite{FK2009}. In those papers, the number of companies is
fixed and strong repulsive drifts are imposed, in order to keep the
configuration of market weights from reaching certain regions of the
unit simplex; the resulting market weights, however, have continuous
paths. Here, by contrast, the number of companies fluctuates due to
breakups and mergers; and the resulting market weights exhibit
discontinuities at such ``event times.'' These differences have a rather
drastic effect: relative arbitrage, which \textit{does} exist with respect
to the market portfolio in \cite{FKK2005} and \cite{FK2009}, is
proscribed here.

\subsection{Preview}

The main results of this paper are as follows. First, we show that
under certain conditions the process that counts the number of
companies is \textit{nonexplosive}: this number does not become infinite
in finite time, so the model can be defined on infinite time horizons.
Second, this model turns out to admit an equivalent martingale measure
by means of a suitable \textsc{Girsanov} transformation: \textit{although
diverse, the model proscribes arbitrage.} This is in contrast with the
models from~\cite{FK2009}, where splits/mergers are not allowed.
Indeed, it was observed in \cite{Strong2011} that in the presence of
splits/mergers, diversity might not lead to arbitrage; in \cite
{FS2011}, \textsc{Strong} and \textsc{Fouque} established this
for their models with a fixed number of companies. We establish the
same result for our model, which allows the number of extant companies
to fluctuate randomly.

The paper is organized as follows. Section~\ref{Info} provides an
informal yet
somewhat detailed description of this model, and states the main
results. Section~\ref{formal} lays out the formal construction of the model.
Section~\ref{proofs} is devoted to the proofs of our results. The
\hyperref[app]{Appendix}
develops a crucial technical result.

\section{Informal construction and main results}
\label{Info}
\subsection{Description of the model}

Consider a stock market with a variable number of companies
\[
X (\cdot)= \bigl\{X(t), 0 \le t < \infty \bigr\},\qquad X(t) = \bigl(X_1(t),
\ldots, X_{\CN(t)}(t) \bigr)^\prime,
\]
where $X_i(t) > 0$ is the capitalization of the company $ i $ at time
$ t \ge0 $, and $\CN(t)$ is the number of companies in the market at
that time. The integer-valued random function $ t \mapsto\CN(t) $
will be piecewise constant; we shall call it the \textit{counting process}
of our model, as it records the number of companies that are extant at
any given time.

At each interval of constancy of this process, the logarithms $ Y_i
(\cdot) =\break  \log X_i (\cdot),  i = 1, \ldots, N $ behave like a system
of Competing Brownian Particles (CBPs) with rank-dependent drifts and
variances. More precisely, the $k$th largest among the $N$ real-valued
processes $ Y_1 (\cdot), \ldots, Y_N(\cdot) $ behaves like a
Brownian motion with local drift $g_{Nk}$ and local variance $ \si
_{Nk}^2 $, for $k = 1, \ldots, N$. These $ g_{Nk} $ and $ \si_{Nk}
> 0 $ with $ N \ge2 , 1 \le k \le N$, are given real constants.
If two or more particles occupy the same position at the same time,
then we break the tie and assign ranks according to the lexicographic
order; more on this in Section~\ref{formal}. We call this model (with constant
number of stocks) a \textit{CBP-based model}.

When the \textit{market weight}
%
\begin{equation}
\label{weights} \mu_i(t) = \frac{X_i(t)}{ \mathcal{C}(t) } ,\qquad \mathcal{C}(t) :=
X_1(t) + \cdots+ X_{\CN(t)}(t)
\end{equation}
of some company $ i = 1,\ldots, \CN(t) $ reaches a given, fixed
threshold $1 - \delta$, a governmental regulatory agency splits this
company into two new companies; one with capitalization $\xi X_i(t)$,
and the other with capitalization $(1 - \xi)X_i(t)$. Here, the random
variable $\xi$ is independent of everything that has happened in the
past, and has a given probability distribution $F$ supported on $[1/2,
1)$; whereas $\delta\in(0, 1/2)$ is a given constant.

In addition, for every integer $ N \ge3 $ there is an exponential
clock with rate $\lambda_N \ge0 $ (a rate of zero means that the clock
never rings); we take formally $\lambda_2=0$, cf. Remark~\ref{Impl}
below. When this clock rings, two companies are chosen at random, as
candidates to be merged and form one new company. The choice is made
according to a certain probability distribution $\mathcal P_N( X(t))$
on the family of subsets of $\{1, \ldots, N\}$ which contain exactly
two elements, and this distribution depends on the current state $X(t)$
of the system. (One example of such dependence is given below, in
Assumption~\ref{ass4}; additional clarification is provided in Section~\ref{Notation}.) If the so-amalgamated company has market weight larger
than or equal to $1 - \delta$, the putative merger is suppressed;
otherwise, the merger is allowed to proceed.

Within the framework of the model thus described in an informal way,
and more formally in Section~\ref{formal} below, we raise and answer
the following questions:
\begin{longlist}[(ii)]
\item[(i)] Are there explosions in this model (i.e., can the number of
companies become infinite in finite time) with positive probability?
Can this model be defined on an infinite time-horizon?

\item[(ii)] What is the concept of a portfolio in this model? Does the
model admit (relative) arbitrage?
\end{longlist}

The answers are described in Theorems \ref{thmm1} and \ref{thmm2} below.

\begin{rmk}
\label{Impl}
We note that this model is free of \textit{implosions}, by its
construction: when there are only two companies, their putative merger
would result in a company with market weight equal to 1 and would thus
be suppressed. This is the reason we took at the outset $ \lambda_2=0 $,
meaning that with only two companies present the merger clock never
rings. As a result, at any given moment there are at least two
companies in the equity market model under consideration; and we need
not specify the rule for picking companies when there are only two of
them, $N = 2$.
\end{rmk}

\subsection{Portfolios and wealth processes}

In the context of the above model, a~\textit{portfolio} is a process
\[
\pi(\cdot) = \bigl\{\pi(t), 0 \le t < \infty \bigr\} , \qquad \pi (t) = \bigl(
\pi_1(t), \ldots, \pi_{\CN(t)}(t) \bigr)^\prime
\]
for which there exists some real constant $K_\pi\ge0$ such that $
|\pi_i(t)| \le K_\pi$ holds for all $0 \le t < \infty$ and $ i =
1, \ldots, \CN(t)$. The quantity $\pi_i(t)$ is called the \textit{portfolio weight} assigned at time $ t $ by the portfolio $\pi(\cdot)$
to the company $ i $; whereas
%
\begin{equation}
\label{MM} \pi_0(t):= 1 - \SL_{i=1}^{\CN(t)}
\pi_i(t) ,\qquad 0 \le t < \infty
\end{equation}
represents the proportion of wealth invested at time $t$ in a money
market with zero interest rate. A portfolio is called \textit{fully
invested}, if it never touches the money market, that is, if $ \pi_0
(\cdot) \equiv0 $; it is called \textit{long-only}, if $ \pi_i ( t)
\ge0 $ holds for all $ i=0, 1, \ldots, \mathcal{N} ( t) $, $ 0
\le t < \infty$.

The prototypical fully invested, long-only portfolio is the \textit{market
portfolio} $ \pi(\cdot) \equiv\mu(\cdot) $ of (\ref{weights}). At
the other extreme stands the \textit{cash portfolio} $ \pi(\cdot)
\equiv\kappa( \cdot) $ with $ \kappa_i ( t) = 0 $ for all $
i=1, \ldots, \mathcal{N} ( t) $, $0 \le t < \infty$; this never
touches the equity market, and keeps all wealth in cash at all times.

When the counting process $ \CN(\cdot) $ jumps up (after a split) or
down (after a merger), the portfolio process behaves as follows:
\begin{longlist}[(ii)]
\item[(i)] if two companies merge, the portfolio weight corresponding to the
new company's stock is the sum of the portfolio weights corresponding
to the two old stocks; whereas

\item[(ii)] if a company gets split into two new ones, its weight in the
portfolio is partitioned in proportion to the weights of the
newly minted companies.
\end{longlist}
The formal description of these actions is postponed to Section~\ref{formal}.

Suppose now that a small investor, whose actions cannot influence asset
prices, starts with initial capital \$1 and invests in the stock market
according to a portfolio rule $\pi(\cdot)$. The corresponding \textit{wealth process} $ V^{\pi} (\cdot)= \{V^{\pi}(t), 0 \le t < \infty\}
$ takes then values in $ (0, \infty) $, satisfies
%
\begin{equation}
\label{wealth} \frac{\mathrm{d}V^{\pi}(t)}{V^{\pi}(t)} = \SL_{i=1}^{\CN(t)}\pi
_i(t) \frac{ \mathrm{d}X_i(t) }{X_i(t)} ,
\end{equation}
and is not affected when the number of companies changes, that is, when
the counting process $ \CN(\cdot) $ jumps. For a derivation of
\eqref{wealth} with a fixed number of companies, see, for instance,
\cite{F2002},
page 6.

\begin{defn}[(Relative arbitrage)] \label{Rel_Arb}
We say that a given portfolio $\pi(\cdot)$ \textit{represents an
arbitrage opportunity relative to another portfolio $\rho(\cdot)$}
over the time horizon $[0,T]$, for some real number $ T >0 $, if we have
%
\begin{equation}
\label{arbitr} \MP \bigl( V^{\pi}(T) \ge V^{\rho}(T) \bigr) =1 ,\qquad
\MP \bigl( V^{\pi
}(T) > V^{\rho}(T) \bigr) >0.
\end{equation}
\end{defn}

In words: over the time-horizon $[0,T]$, the portfolio $\pi(\cdot) $
performs at least as well as $\rho(\cdot)$ with probability one, and
strictly better with positive probability. When $ \rho(\cdot) \equiv
\kappa(\cdot) $ is the cash portfolio, this reduces to the usual
definition of arbitrage.

\subsection{Main results}

Let us impose some conditions on the parameters of this model. A
salient feature of real-world markets is that stocks with smaller
market weights tend to have larger drift coefficients (growth rates),
so it is not unreasonable to impose the following condition.

\begin{asmp}
\label{condition}
\[
g_{N1} \le \min_{2 \le k \le N}g_{Nk} \qquad\mbox{for
every } N \ge2.
\]
\end{asmp}

We shall also impose the following conditions: there exist constants
$\ol{\si}, \underline{\si}, \ol{g}$ such that we have the following.

\begin{asmp}\label{ass2}
We have $\delta\in(0, 1/6)$, as well as
\[
0 < \underline{\si} \le\si_{Nk} \le\ol{\si} < \infty,\qquad
|g_{Nk}| \le \ol{g} \qquad\mbox{for all } N \ge2 \mbox{ and } k = 1, \ldots, N.
\]
\end{asmp}

\begin{asmp}\label{ass3}
The probability distribution $F$ of the random variable $ \xi$
responsible for splitting companies is supported on the interval $[1/2,
1-\eps_0]$, where $\eps_0 \in( 0, 1/2)$. In other words,
%
\begin{equation}
\label{esssup} \xi\backsim F \quad\Ra\quad 1/2 \le\ess\inf\xi\le\ess\sup\xi\le1-
\varepsilon_0.
\end{equation}
\end{asmp}

\begin{asmp}\label{ass4} The rule for picking companies to be merged is as follows:
With $N \ge3$, we exclude the company which occupies the highest rank
in terms of capitalization and choose at random two of the remaining $
N-1 $ companies according to the uniform distribution over the
%
\begin{equation}
\label{mN} m_N =\pmatrix{N-1
\cr
2}
\end{equation}
possible such choices. If two or more companies are tied in terms of
capitalization, we resolve the tie \textit{lexicographically}, that is,
always in favor of the lowest index $i$.
\end{asmp}

\begin{rmk} Under Assumptions \ref{ass2} and \ref{ass4}, mergers are
never suppressed.
Indeed, of the chosen companies, the one with the biggest market weight
will occupy the second place at best, so its market weight will be no
more than $1/2$; whereas the other will occupy the third place at best,
so its market weight will not exceed $1/3$. Therefore, the market
weight of the amalgamated company will not exceed $5/6$, a number
smaller than $1 - \delta$ because we have $ \delta< 1/6 $ from Assumption~\ref{ass2}, \textit{and so the merger will not be suppressed.}
Moreover, \textit{all}
of the new market weights will be bounded away from the threshold $1 -
\delta$, so it will take some time for \textit{any} company extant
after the
merger to hit this threshold.
\label{rmk:never-suppressed}
\end{rmk}

\begin{asmp}\label{ass5}
The rates of the exponential clocks satisfy, for some real constants
$c $ and $ \al> 0$,
%
\begin{equation}
\label{la} \lambda_N \asymp c N^{\al} ,\qquad N \to\infty.
\end{equation}
\end{asmp}

%
\begin{rmk} This condition is perhaps the most significant one: it
ensures that mergers happen with sufficient intensity, so that the
number of companies in the model will not only not become infinite in a
finite amount of time, but will also have a ``tame'' temporal growth
(cf. Proposition~\ref{tails} below).

As an illustration for Condition (\ref{la}), suppose there are $N$
companies; then, according to the rules of Assumption~\ref{ass4},
there are $
m_N $ such possible mergers as in~(\ref{mN}). If each pair of
companies has its own merger exponential clock $ \Xi_i $ with the
same parameter $\lambda$, and if $ \Xi_1, \ldots, \Xi_{m_N}$ are
independent, then the earliest merger will happen at the smallest of
those exponential clocks; but
\[
\min (\Xi_1, \ldots, \Xi_{m_N} ) \backsim
\CE(m_N\lambda) ,
\]
so $\lambda_N = m_N\lambda\asymp N^2$ as $N \to\infty$. That is, the
requirement (\ref{la}) holds in this case with $\al=2$.
\end{rmk}

The following two theorems are our main results. They are proved in
Section~\ref{proofs}.

\begin{thmm}
\label{thmm1}
Under Assumptions \ref{condition}--\ref{ass5}, the above market model is
free of explosions and
can thus be defined on an infinite time-horizon.
\end{thmm}

\begin{thmm} Under the Assumptions \ref{condition}--\ref{ass5}, no
relative arbitrage is
possible over any given time horizon $[0,T]$ of finite length.
\label{thmm2}
\end{thmm}

\section{Formal construction}
\label{formal}

\subsection{Notation}
\label{Notation}
We let $\mathbb{N}_0:= \{0, 1, 2, \ldots\}$. For every integer $N \ge
2, \delta\in(0, 1)$, we let
\begin{eqnarray*}
\De_+^N&:=& \bigl\{(z_1, \ldots, z_N) \in
\BR^N\mid z_1>0 , \ldots, z_N > 0 ,
z_1 + \cdots+ z_N = 1 \bigr\} ,
\\
\De_+^{N, \delta}&:=& \bigl\{(z_1, \ldots, z_N) \in
\De_+^N \mid z_1\le 1 - \delta, \ldots, z_N
\le1 - \delta \bigr\}.
\end{eqnarray*}
We also denote
%
\begin{eqnarray}
\label{SMM} \mathcal S&:=& \bigcup_{N=2}^{\infty}(0,
\infty)^{N } ,\qquad \widetilde{\mathcal S}:= \bigcup
_{N=3}^{\infty}(0,\infty)^{N } ,
\nonumber
\\[-8pt]
\\[-8pt]
\nonumber
 \mathcal M&:=&
\bigcup_{N=2}^{\infty}\De_+^{N } ,\qquad
\mathcal M^{\delta}:= \bigcup_{N=2}^{\infty}
\De_+^{N, \delta}.
\end{eqnarray}
For $x \in\mathcal{S}$, we denote by $\mathfrak N(x)$ the number of
components of $x$, that is, the integer $ N \ge2 $ for which $ x
\in(0, \infty)^N$; and for $N= \mathfrak N(x)$, we denote by $
\mathfrak{z} (x) \in\De_+^N $ the vector with components
\[
\mathfrak{z}_i (x):= \frac{x_i}{ x_1 + \cdots+ x_N} ,\qquad i=1, \ldots, N.
\]
The market-weight process $\mu(\cdot) = \mathfrak{z}  (X(\cdot
)
)$ of (\ref{weights}) is said to be \textit{on the level $N$} at time $t$,
if $\mu(t) \in\De^{N }_+$.

For any given vector $y \in\BR^N$, we denote by $y_{(1)} \ge y_{(2)}
\ge\cdots\ge y_{(N)}$ its \textit{ranked components}; in this ranking,
ties are resolved lexicographically, always in favor of the lowest
index as in \cite{BFK2005,Ichiba11} and in Assumption~\ref{ass4}.

We denote by $C^r(A)$ the set of $r$ times continuously differentiable %
functions $f: A \to\BR$.

For every integer $N \ge3 $, we denote by $\mathcal R_N $ the family
of subsets of $\{1, \ldots, N\}$ which contain exactly two elements. We
denote by $\mathscr Q_N$ the collection of all probability
distributions on $\mathcal R_N$.

\begin{rmk}
\label{clarify}
Under the Assumption~\ref{ass4}, the probability distributions\break $ \{
\mathcal
P_{\mathfrak N(x)}( x)\}_{x \in\widetilde{\mathcal{S}}} $ in Section~\ref{Info} are constructed thus:
For any given $x \in\widetilde{\mathcal{S,}}$ we let $N=\mathfrak
N(x)$, rank lexicographically the components of the vector $x$, and
consider the smallest index $j \in\{ 1, \ldots, N\}$ such that $x_j
\ge x_k $ holds for all $ k = 1, \ldots, N $. Then $\mathcal P_N( x)
\in\mathscr Q_N$ is the uniform distribution on
the family of subsets of $\{1, \ldots, N\} \setminus\{ j\} $ that
contain exactly two elements [there are exactly $m_N$ such subsets, as
in (\ref{mN})].
\end{rmk}

\subsection{Fixed number of companies}
\label{FNC}

First, let us formally introduce auxiliary CBP-based models with a
fixed, finite number of particles. These will serve as building blocks
for our ultimate model; in \cite{FS2011}, similar preparatory models
are referred to as ``premodels.''

Fix an integer $N \ge2$, and consider a system of $N$ particles moving
on the real line, formally expressed as an 
$\BR^N$-valued process
\[
Y (\cdot)= \bigl\{Y(t), 0 \le t < \infty \bigr\} ,\qquad Y(t) = \bigl(Y_1(t),
\ldots, Y_N(t) \bigr)^\prime.
\]
Consider a filtered probability space $ (\Oa, \CF, \MP ),
\mathbb{G}= \{\CG(t)\}_{0 \le t < \infty} $, where the filtration
satisfies the \textit{usual conditions} of right-continuity and
augmentation by null sets, and let $W (\cdot) = \{W(t), 0 \le t <
\infty\}$ be a standard $N$-dimensional $(\mathbb{G} , \MP
)$-Brownian motion.

\begin{defn}
\label{CBP}
With $g_1, \ldots, g_N$ given real numbers, and $\si_1, \ldots, \si_N$
given positive real numbers, a \textit{finite system of CBPs with
symmetric collisions} is an $\BR^N$-valued process governed by the
system of stochastic differential equations
\[
\mathrm{d}Y_i(t) = \SL_{k=1}^N \mathbf{
1}_{\{Y_i(t) =
Y_{(k)}(t)\}} \bigl(g_k \,\mathrm{d}t + \si_k
\,\mathrm{d}W_i (t) \bigr) ,\qquad  i = 1, \ldots, N, 0 \le t < \infty.
\]
\end{defn}

Informally, such a model posits that the $k$th largest particle moves
as a one-dimensional Brownian motion with local drift $g_k$ and local
variance $\si_k^2$. We denote the ranked (in decreasing order)
statistics for the components of this system as
%
\begin{equation}
\label{ranks} \max_{1 \le i \le N} Y_i (\cdot) =:
Y_{(1)} (\cdot) \ge Y_{(2)} (\cdot ) \ge\cdots\ge
Y_{(N)} (\cdot):= \min_{1 \le i \le N} Y_i (\cdot
).
\end{equation}
Similarly, $\mu_{(k)}(t)$ is the $k$th ranked market weight at time $t
$: namely, $\mu_{(1)}(t) \ge\cdots\ge\mu_{(N)}(t)$.
We set $\La_{(k, k+1)} (\cdot) = \{ \La_{(k, k+1)}(t), 0 \le t <
\infty\}$ for the local time accumulated at the origin by the
nonnegative semimartingales $Y_{(k)} (\cdot)- Y_{(k+1)} (\cdot) =
 \{
Y_{(k)}(t) - Y_{(k+1)}(t), 0 \le t < \infty \}$ with $ k=1,
\ldots, N-1$ (for notational convenience, we set also $\La_{(0,
1)}(\cdot) \equiv\La_{(N, N+1)}(\cdot) \equiv0$ for all $t \in[ 0,
\infty)$). Then the equation
%
\begin{eqnarray}
\label{localtimes} \mathrm{d}Y_{(k)}(t) = g_k\,\mathrm{d}t +
\si_k\,\mathrm{d}B_k(t) + \tfrac{1}2 \,\mathrm{d}
\La_{(k, k+1)}(t) - \tfrac{1}2 \,\mathrm{d} \La_{(k-1,
k)}(t) ,
\nonumber
\\[-8pt]
\\[-8pt]
\eqntext{0
\le t < \infty}
\end{eqnarray}
describes the dynamics of the ranked semimartingales in (\ref{ranks}),
where the standard Brownian motions
%
\begin{equation}
\label{Levy} B_k (\cdot):= \sum_{i=1}^N
\int_0^{ \cdot}\mathbf{ 1}_{\{
Y_i(t) = Y_{(k)}(t)\}} \,\mathrm{d}
W_i (t) ,\qquad k=1, \ldots, N
\end{equation}
are independent by the P. \textsc{L\'evy} theorem. We refer to \cite
{BFK2005}, \cite{Ichiba11}, Lemma~1 and \cite{IchibaThesis}, Section~3,
for the derivation of (\ref{localtimes}), as well as to \cite{Bass1987}
for the existence and uniqueness in distribution of a weak solution to
the CBP system of Definition~\ref{CBP}. As shown in \cite{IKS2013},
pathwise uniqueness, and thus existence of a strong solution, also hold
for this system up until the first time three particles collide---and
the latter never happens if the mapping $ k \mapsto\si^2_k $ is
concave (cf. \cite{IKS2013,MyOwn5,MyOwn3,IK2010}).

The \textit{CBP-based market model} with a fixed number $N$ of companies
is defined as a collection of $N$ real-valued, strictly positive
stochastic processes
$
X_i (\cdot)=  \{X_i(t), 0 \le t < \infty \} , i = 1, \ldots
, N$ with $ X_i(t):= e^{Y_i(t)} $. The dynamics of these processes
are given by
%
\begin{equation}
\label{logX} \mathrm{d} \log X_i(t) = \sum
_{k=1}^N \mathbf{ 1}_{\{X_i(t) =
X_{(k)}(t)\}} \bigl[
g_k \,\mathrm{d}t + \si_k \,\mathrm{d}W_i (t)
\bigr] ,
\end{equation}
or equivalently
%
\begin{equation}
\label{X} \frac{ \mathrm{d} X_i(t) }{X_i(t)} = \sum_{k=1}^N
\mathbf{ 1}_{\{X_i(t) = X_{(k)}(t)\}} \biggl[ \biggl( g_k + \frac{ \sigma_k^2
}{2}
\biggr) \,\mathrm{d}t + \si_k \,\mathrm{d}W_i (t) \biggr].
\end{equation}
In this model, the vector process $ \mu(\cdot) =  ( \mu_1 (\cdot
), \ldots, \mu_N (\cdot)  )' = \mathfrak{z}  (X(\cdot)
) $
of the market weights
$
\mu_i (t ):= X_i (t)/  ( X_1 (t) + \cdots+ X_N ( t) ) $ with
$ i=1, \ldots, N , 0 \le t < \infty
$ for its various companies, evolves as an $N$-dimensional diffusion
governed by the system of SDEs
%
\begin{eqnarray}
\label{dyna} \mathrm{d}\log\mu_i(t) &=& \biggl[
\SL_{k=1}^Ng_k \mathbf{ 1}_{\{
\mu
_i(t) = \mu_{(k)}(t)\}} -
\SL_{k=1}^N g_k \mu_{(k)} (t) \nonumber\\
&&{}-
\frac{1}2\SL_{k=1}^N \si_k^2
\bigl(\mu_{(k)}(t) - \mu_{(k)}^2(t) \bigr)
\biggr]\,\mathrm{d}t
\nonumber\\
&&{} + \SL_{k=1}^N \si_k \bigl[ \mathbf{
1}_{\{\mu_i(t) = \mu_{(k)}(t)\}} \,\mathrm{d}W_i(t) \\
&&{}- \mu_{(k)} (t)
\SL_{\nu= 1}^N \mathbf{ 1}_{\{\mu_{\nu}(t) = \mu
_{(k)}(t)\}}\,\mathrm{d}W_\nu(t)
\bigr],\nonumber\\
\eqntext{i=1, \ldots, N.}
\end{eqnarray}

We derive this system from the general expression in equation (2.4) of
\cite{FK2009}, Section~2. Substituting in that expression the concrete
values of drift and covariance coefficients for the CBP-based market
model of (\ref{logX}) under consideration, we arrive at the dynamics of
(\ref{dyna}) for the $ \log\mu_{i}(\cdot)$'s.

\begin{rmk}
In the terminology of \cite{FS2011}, Remark~2, the companies in models
of this sort are ``generic'': The characteristics of their
capitalizations' dynamics depend entirely on the ranks the companies
occupy in the capitalization hierarchy; they are not idiosyncratic
(i.e., name- or sector-dependent).
\end{rmk}

\subsection{Formal construction of the main model}
\label{formcon}

Let us begin the formal construction of our model. This will take the
form of a process $X (\cdot)=  \{X(t), 0 \le t < \infty \}$ on
the state-space $\mathcal{S}$ of (\ref{SMM}).

For every $N \ge2$ and every $x = (x_1, \ldots, x_N)' \in(0, \infty
)^N$, we construct a probability space $(\Oa^{N, x}, \CF^{N, x}, \MP
^{N, x})$ which contains countably many i.i.d. 
copies $ Y^{N, x, n} (\cdot) , n \in\mathbb{N} $ of the solution
\[
Y^{N, x} (\cdot) = \bigl\{Y^{N, x}(t), 0 \le t < \infty \bigr\}
, \qquad Y^{N, x}(t) = \bigl(Y_1^{N, x}(t), \ldots,
Y_N^{N,
x}(t) \bigr)^\prime
\]
to the following system of stochastic differential equations:
%
\begin{eqnarray}
\label{fixed} \mathrm{d}Y_i^{ N, x }(t) =
\SL_{k=1}^N\mathbf{ 1}_{\{ Y_i^{N, x}(t) =
Y_{(k)}^{N, x}(t)\}} \bigl(
g_{Nk} \,\mathrm{d}t + \si_{Nk} \,\mathrm{d}W_i(t)
\bigr) ,
\nonumber
\\[-8pt]
\\[-8pt]
\eqntext{ Y^{N, x}_i(0) = \log x_i}
\end{eqnarray}
for $ i=1, \ldots, N $. Here $ W(\cdot) = \{W(t), t \ge0\}$ is a
standard $N$-dimensional Brownian motion, and the parameters $ g_{Nk}
$, $ \si_{Nk} $ satisfy the conditions of Assumptions \ref{condition} and
\ref{ass2}.

For every $N \ge3$, we fix a collection of probability distributions\break 
$ \{\mathcal P_N( x)\}_{x \in\widetilde{\mathcal{S}}} \subseteq
\mathscr Q_N $. This specification will provide the rule for choosing
two out of the existing $N = \mathfrak N(x)$ companies to merge, when
the system is in state $ x \in\widetilde{\mathcal{S}} $ and an
exponential clock rings.

\begin{itemize}
\item Consider another probability space $(\Oa', \CF', \MP')$ which contains:
\begin{longlist}[(a)]
\item[(a)] countably many countably many i.i.d. copies $\xi_1, \xi_2, \xi_3,
\ldots$ of a random variable $ \xi$ with given probability
distribution $F$, which is supported on the interval $[1/2, 1-\eps_0]$;

\item[(b)] for each $N \ge2$, countably many copies $\eta^N_1, \eta^N_2,
\ldots$ of an exponential clock $\eta^N$ with rate $\lambda_N$, if this
rate is positive (if the rate is zero, as we assume it is for $N=2$, we
let $\eta^N_1 = \eta^N_2 = \cdots= \infty$);

\item[(c)] for each $N \ge2$ and each probability distribution $\mathfrak p
\in\mathscr Q_N$, countably many i.i.d. copies $\zeta_i(N,
\mathfrak p), i = 1, 2, \ldots$ of a random element $\zeta(N,
\mathfrak p) $, which takes values in $\mathcal R_N$
and is distributed according to $ \mathfrak p $. (Please recall here
the notation of Section~\ref{Notation}.)
\end{longlist}
\item
Let $(\Oa, \CF, \MP)$ be the direct product of these probability
spaces. Starting from a point $ X(0)= x \in\mathcal{S}$, we let $N_0
= \mathfrak N(x)$ be the number of companies extant at $t=0$, and
construct a process $X (\cdot) =  \{X(t), 0 \le t < T  \}$ and
a random time $T $, the ``lifetime'' of $X(\cdot)$, as follows:
\begin{longlist}[(iii)]
\item[Step (i):] For $t \le\tau_1\wedge\eta_1^{N_0}$, we define the
random vector $X(t)= (X_1 (t), \ldots,\break   X_{N_0} (t))^\prime$ with
values in $ (0, \infty)^{N_0}$ as
%
\begin{eqnarray}
\label{T1} X_i(t)&:=& \exp \bigl(Y_i^{N_0, x, 1}(t)
\bigr)\quad \mbox{and}
\nonumber
\\[-8pt]
\\[-8pt]
\nonumber
 \tau _1&:= &\inf \bigl\{ t \ge0\mid\exists i = 1,
\ldots, N_0: \mu ^{N_0, x}_i(t) = 1 - \delta \bigr
\}
\end{eqnarray}
(we adopt here the usual convention $\inf\varnothing= \infty$), where
the market weight of the company $i$ is
%
\begin{eqnarray}
\label{mu} \mu_i (t) \equiv\mu_i^{N_0, x}(t):
= \frac{X_i(t)}{ X_1(t)+ \cdots+
X_{N_0} (t) } ,
\nonumber
\\[-8pt]
\\[-8pt]
\eqntext{0 \le t \le\tau_1\wedge\eta_1^{N_0};
i = 1, \ldots, N_0.}
\end{eqnarray}
Since $\delta\in(0, 1/2)$, there can be at most one index $i$ with $
\mu
^{N_0, x}_i(\tau_1) = 1 - \delta $. Thus, the moment of the first jump,
or ``event,'' in this setup, is
\[
T_1:= \tau_1\wedge\eta_1^{N_0}.
\]

\item[Step (ii):] If $ \tau_1 \le\eta^{N_0 }_1, \tau_1<\infty$, we
pick the unique $ i \in\{ 1, \ldots, N_0 \} $ such that $\mu_i(\tau
_1) = 1 - \delta$, and define the vector $X(\tau_1 +) \in(0, \infty
)^{N_0+1}$ as follows:
\begin{eqnarray*}
X_\nu(\tau_1 +) &=& X_\nu(\tau_1
),\qquad \nu= 1, \ldots, i - 1;\\
 X_\nu(\tau_1 +) &=&
X_{\nu+1}(\tau_1 ),\qquad \nu= i, \ldots, N_0-1;
\\
X_{N_0}(\tau_1+) &=& \xi_1 X_i(
\tau_1 ) ,\qquad X_{N_0+1}(\tau_1+) = (1 -
\xi_1)X_i(\tau_1 ).
\end{eqnarray*}

To wit: at the time $ \tau_1 $ of (\ref{T1}), company $ i $ is
split into two new companies, anointed with the names $ N_0 $ and $
N_0 +1 $. These inherit the capitalization $ X_i(\tau_1 ) $ of
their progenitor in proportions $ \xi_1 $ and $ 1 - \xi_1 $,
respectively. Companies $ 1, \ldots, i-1 $ keep both their names and
their capitalizations; whereas the companies formerly known as $ i+1,
\ldots, N_0 $ keep their capitalizations but change their names to $
i , \ldots, N_0 -1 $, respectively.

\item[Step (iii):] If $\tau_1 > \eta^{N_0}_1$, a subset with two
elements $\{i, j\} \subseteq\{1, \ldots, N_0\}$ is selected according
to the random variable $\zeta_1  (N_0, \mathcal P_{N_0}(X(\eta
^{N_0}_1)) )$ whose distribution is $ \mathcal{P}_{N_0}(X(\eta
^{N_0}_1)) \in\mathscr{Q}_{N_0} $.

On the event $  \{\mu_i(\eta_1^{N_0}) + \mu_j(\eta_1^{N_0}) \ge1
- \delta \}$, we proceed to step (iv), case B below. Otherwise, we
define the vector $X(\eta^{N_0}_1 +) \in(0, \infty)^{N_0-1}$ as
follows, say with $i < j $:
\begin{eqnarray*}
X_\nu\bigl(\eta^{N_0}_1 +\bigr) &=&
X_\nu\bigl(\eta^{N_0}_1 \bigr),\qquad \nu= 1, \ldots,
i - 1 ;\\
 X_\nu\bigl(\eta^{N_0}_1 +\bigr) &=&
X_{\nu+1}\bigl(\eta^{N_0}_1 \bigr),\qquad \nu= i, \ldots,
j - 2;
\\
X_\nu \bigl(\eta^{N_0}_1 + \bigr) &=&
X_{\nu+2} \bigl(\eta^{N_0}_1 \bigr),\qquad \nu = j-1,
\ldots, N_0-2;\\
 X_{N_0-1} \bigl(\eta^{N_0}_1
+ \bigr) &=& X_i \bigl(\eta^{N_0}_1 \bigr) +
X_j \bigl(\eta^{N_0}_1 \bigr).
\end{eqnarray*}

Once again, companies $ 1, \ldots, i-1 $ keep both their names and
their capitalizations. The erstwhile companies $ i+1, \ldots, j-1 $
keep their capitalizations but change their names to $ i , \ldots,
j-2 $; whereas the erstwhile companies $ j+1, \ldots, N_0 $ keep
their capitalizations but change their names to $ j-1 , \ldots,
N_0-2 $. The former companies $ i $ and $ j $ merge; they create a
new company, anointed with the index (name) $ N_0 -1 $, which
inherits the sum total of their capitalizations.\vspace*{6pt}

\item[Step (iv):] We let $N_1 = \CN( T_1+ )$ be the new number of
companies extant right after time $T_1= \tau_1 \wedge\eta_1^{N_0} $,
and note that there are three possibilities:

\textit{Case} A: $ N_1 =
N_0+1 $ on the event $ \{ \tau_1 \le\eta^{N_0}_1, \tau_1<\infty\}
$ of a split;

\textit{Case} B: $ N_1 = N_0 $ on the event $ \{
\tau_1 > \eta^{N_0}_1, \mu_i(\eta_1^{N_0}) + \mu_j(\eta_1^{N_0})
\ge
1 - \delta\}$ of a ``suppressed'' merger;

\textit{Case} C: $ N_1 = N_0 -1 $ on the event $ \{ \tau_1 > \eta
^{N_0}_1, \mu_i(\eta_1^{N_0}) + \mu_j(\eta_1^{N_0}) < 1 - \delta\}
$ of
a ``successful'' merger.

We define
%
\begin{equation}
\label{T2} X_i(t):= \exp \bigl(Y_i^{N_1, x_1, 2}
(t - T_1 ) \bigr) \qquad\mbox{for } T_1 < t \le
T_1 + \bigl( \tau_2 \wedge\eta^{N_1}_2
\bigr).
\end{equation}
Here, $T_1= \tau_1 \wedge\eta_1^{N_0} , x_1 = X(T_1)$, and
\[
\tau_2:= \inf \bigl\{t > 0 \mid\exists i = 1, \ldots,
N_1: \mu ^{N_1, x_1}_i(T_1+ t) = 1 -
\delta \bigr\} ,
\]
where $\mu^{N_1, x_1}_i(t)$ is defined by analogy with \eqref{mu}, in
terms of the capitalizations in (\ref{T2}), as
\begin{eqnarray}
\mu_i (t) \equiv\mu_i^{N_1, x_1}(t):=
\frac{X_i(t)}{ X_1(t)+ \cdots
+ X_{N_1} (t) } ,
\nonumber
\\
\eqntext{T_1 < t \le T_1 + \bigl(
\tau_2 \wedge\eta^{N_1}_2 \bigr) , i = 1, \ldots,
N_1.}
\end{eqnarray}
Thus, the time of the second jump in the integer-valued process $
\mathcal{N} (\cdot) $ is
\[
T_2:= T_1 + \bigl( \tau_2 \wedge
\eta^{N_1}_2 \bigr).
\]

\item[Step (v):] We similarly define  the values of the capitalization
processes after the second jump. On each next step, we use new
independent copies of variables $\eta^N_i$ and $\xi_i$. [If this jump
corresponds to a merger, then we choose two companies to be merged
according to the distribution $\zeta_2(N_1,\break  \mathcal P_{N_1}(X(\eta
^{N_1}_2)))$.] Then we define their evolution until the moment $T_3$ of
the third jump, etc.

\begin{rmk}
We already saw in Remark~\ref{rmk:never-suppressed} that the
specification of the probability distributions $ \{\mathcal
P_{\mathfrak N(x)}( x)\}_{x \in\widetilde{\mathcal{S}}} $ as in
Assumption~\ref{ass4} and Remark~\ref{clarify} guarantees that, with
three or
more companies present, no merger is ever suppressed [i.e., that case B
in (iv) above never occurs].
\end{rmk}

\item[Step (vi):] This construction leads to a strong \textsc{Markov}
process $X (\cdot)=  \{ X(t),  0 \le t < T_\star \}$ with state
space $\mathcal{S}$ and piecewise-continuous, LCRL (Left-Continuous
with Right Limits) paths, defined on the time interval $[0, T_\star)$ with
%
\begin{equation}
\label{expl} T_\star:= \lim_{m \to\infty} \uparrow
T_m.
\end{equation}
\end{longlist}
\end{itemize}

The resulting market-weight process $\mu(\cdot) = \mathfrak{z}
(X(\cdot) ) $ of (\ref{weights}) has state space $\CM^{\delta}$
as in
(\ref{SMM}). In particular, $ \max_{1 \le i \le\mathcal{N} (t)} \mu
_i (t) \le1 - \delta$ holds for all $ t \in[0, \infty) $, so the
resulting market model is \textit{diverse} in the terminology of \cite
{F2002}, Chapter~2. We also note that, in all cases of the above
construction, the total capitalization $ \mathcal{C}(\cdot) $ in
(\ref
{weights}) is preserved at each ``event-time'' $ T_m $, namely
$
\mathcal{C}(T_m +) = \mathcal{C}(T_m ) , \forall m \in
\mathbb{N} $.

\begin{defn}
\label{EXP}
We say that the so-constructed model \textit{admits explosions}, if $
\MP( T_\star= \infty) <1 $. Otherwise, the model is said to be
\textit{free of explosions}.
\end{defn}

Theorem~\ref{thmm1} guarantees that $ \MP( T_\star= \infty) =1
$ holds under Assumptions \ref{condition}--\ref{ass5}. In the absence of
explosions, the
process $X (\cdot) $ is defined on all of $[0, \infty)$, and we denote
by $ \mathbb{F} = \{ \mathcal{F} (t) \}_{0 \le t < \infty} $ the
smallest filtration to which $X(\cdot)$ is adapted and which satisfies
the usual conditions.

\begin{rmk}
In addition to being diverse, the model just constructed has intrinsic
relative variance (equivalently, excess growth rate for market
portfolio) which is bounded away from zero, namely
\begin{eqnarray}
\gamma^\mu_* (t):= \frac{1}2 \Biggl( \sum
_{k=1}^{N} \sigma_{N
k}^2
\mu_{(k)} (t) \bigl( 1-\mu_{(k)} (t) \bigr) \Biggr)
\bigg|_{N=\mathcal{N}(t)} \ge \bigl( \si^2_0 \delta \bigr) / 2 >
0 ;\nonumber\\
\eqntext{ 0 \le t < \infty.}
\end{eqnarray}
We owe this observation to Dr. Robert \textsc{Fernholz}. See
Proposition~3.1 in \cite{FK2005}, or Example~11.1 in \cite{FK2009}, for
the significance of such a positive lower bound in the context of
arbitrage relative to the market portfolio with a fixed number of companies.
\end{rmk}

\subsection{Portfolios and associated wealth processes}

Let us discuss portfolios and the wealth processes they generate. A
\textit{portfolio} $ \pi(\cdot) $ is an $ \mathbb
{F}$-progressively measurable process
$
\pi(\cdot)=  \{ \pi(t), 0 \le t < \infty \} , \pi(t) =
 (\pi_1(t), \ldots, \pi_{\CN(t)}(t) )^\prime$ for which there
exists some real constant $K_\pi\ge0$ such that, almost surely: $
|\pi_i(t)| \le K_\pi$ holds for all $0 \le t < \infty$ and $ i =
1, \ldots, \CN(t)$.

When the integer-valued process $ \CN(\cdot) $ suffers a downward or
upward jump, this portfolio must behave according to the rules
described informally in Section~\ref{Info}. We formalize these rules presently.
\begin{longlist}[(A)]
\item[(A)] Assume that at time $t$, the $i$th and $j$th companies merge
into one company; this is then anointed with index $N-1$, where $N =
\CN
(t )$ 
is the number of companies immediately before the merger. The new
portfolio weights are
\begin{eqnarray*}
\pi_\nu(t+) &=& \pi_\nu(t ),\qquad \nu= 1, \ldots, i-1;\\
\pi_\nu (t +) &=& \pi_{\nu+1}(t ),\qquad \nu= i, \ldots, j-2;
\\
\pi_\nu(t+) &=& \pi_{\nu+2}(t ), \qquad\nu= j-1, \ldots, N - 2;\qquad
\pi_{N -
1}(t+) = \pi_i(t ) + \pi_j(t ).
\end{eqnarray*}
In words: companies not involved in the merger are assigned the same
portfolio weights, under their new appellations if necessary; whereas
the newly minted company $ N-1 $ inherits the sum of the portfolio
weights formerly assigned to its two parent companies.

\item[(B)] Assume that at time $t$ the $i$th company, with
capitalization $X_i(t )$, is split into two companies [anointed with
the indices $N$ and $N+1$, where $N = \mathcal{N} (t)$ 
is the number of companies immediately before the split]. The new
portfolio weights are
\begin{eqnarray*}
\pi_\nu(t+) &=& \pi_\nu(t ),\qquad \nu= 1, \ldots, i-1;\\
 \pi
_\nu(t+) &=& \pi_{\nu+1}(t ), \qquad\nu= i, \ldots, N-1;
\\
\pi_{N}(t+)& =& \pi_i(t )\frac{X_N(t+)}{X_i(t )} ,\qquad \pi
_{N+1}(t+) = \pi_i(t )\frac{X_{N+1}(t+)}{X_i(t )}.
\end{eqnarray*}
Once again, companies not involved in the split keep their weights in
the portfolio, under their new appellations if necessary; whereas each
of the two newly created companies $ N $ and $ N+1 $ inherits the
weight in the portfolio of the parent company, in proportion to its
currently assigned capitalization.
\end{longlist}

The corresponding wealth process $V^{\pi} (\cdot) = \{V^{\pi}(t), 0
\le t < \infty\}$ is continuous: \textit{it does not suffer a jump
when a
split or a merger happen.} It is $ \mathbb{F}$-adapted, takes %
values in $ (0, \infty) $, and is governed for each integer $ m \in
\mathbb{N}_0 $ by the dynamics
%
\begin{equation}
\label{wealthSDE} \frac{\mathrm{d}V^{\pi}(t)}{V^{\pi}(t)} = \SL_{i=1}^{N_m}
\pi_i (t) \frac{\mathrm{d}X_i(t)}{X_i(t)} ,\qquad t \in(T_m,
T_{m+1}) , V^\pi(0) =1
\end{equation}
and with $T_0=0$. Quite clearly, we have $ V^\kappa(\cdot) \equiv1
$ for the cash portfolio; and $ V^\mu(\cdot) \equiv\mathcal
{C}(\cdot
) / \mathcal{C}(0) $ for the market portfolio of (\ref{weights}). As
mentioned before, the amount $
\pi_0(t)V^{\pi}(t) $ in the notation of (\ref{MM}) is invested in
the money market at time~$t$.

\section{Proofs}
\label{proofs}

\subsection{Subexponential tail}\label{sec4.1}

We state and prove the following crucial proposition. This result
postulates that the distribution of the maximum number of companies
over any finite time-interval has a tail which is lighter than that of
any exponential distribution.

\begin{prop}
\label{tails}
Under Assumptions \ref{condition}--\ref{ass5}, for any $T \in( 0, \infty
)$, we have
\[
\lim_{u \to\infty} \frac{1}u \Bigl[-\log \MP_x
\Bigl( \max_{0 \le t \le T}\CN(t) > u \Bigr) \Bigr] = \infty.
\]
In particular, 
for all $ c \in( 0, \infty) $ we have
%
\begin{equation}
\label{subexp} \mathbf E_x \Bigl[ \exp \Bigl(c\max
_{0 \le t \le T}\CN(t) \Bigr) \Bigr] < \infty.
\end{equation}
%
\end{prop}

Theorem~\ref{thmm1} follows directly from this proposition: If the
maximal number of companies over the time-horizon $[0, T]$ has this
property, then it is a.s. finite, which is another way of saying
that the counting process $ \mathcal{N} (\cdot) $ does not explode.
Theorem~\ref{thmm2} also uses this fact, but in subtler ways; its proof
is postponed until Section~\ref{proof_thmm2}.

The rest of this section is organized as follows. In Section~\ref
{sec: overview}, we
explain the main idea of the proof of Proposition~\ref{tails}. In
Section~\ref{Prelims}, we derive some preliminary estimates and lay the
groundwork for the rest of the proof. In Section~\ref{proof_Prop}, we
carry out
the proof of Proposition~\ref{tails} in full detail. In Section~\ref{proof_thmm2},
we use this Proposition to prove Theorem~\ref{thmm2}.

\subsection{Overview of the proof of Proposition \texorpdfstring{\protect\ref{tails}}{4.1}}
\label{sec: overview}

We shall use the following notation: Consider the random sequence $\{
N_m\}_{m \in\mathbb{N}_0} $ with $N_m = \mathcal{N} (T_m +)$, and
the sequence of ``event-times''\vadjust{\goodbreak} $\{ T_m\}_{m \in\mathbb{N}_0} $, as
in (\ref{expl}) and (\ref{wealthSDE}). The quantity $N_m$ is the level
of the process $\mu(\cdot)$ of market weights; in other words, the
number of companies extant during the time interval $(T_m, T_{m+1})$
between the $m$th and the $(m+1)$st jumps of the integer-valued process
$ \mathcal{N}(\cdot) $.

The idea of the proof of Proposition~\ref{tails} is as follows. We say
that a \textit{double jump upward} happens at the
step $m$, if $N_{m} = N_{m-1}+1$ and $N_{m+1} = N_m+1$. If $N_m = N$,
then we say that this is a \textit{double jump upward from level $N-1$ to
level $N+1$} at step $m$; we denote this event by
%
\begin{equation}
\label{eq:double-jump} A(m, N):= \{N_{m+1} = N_m + 1,
N_m = N_{m-1} + 1\}.
\end{equation}
Suppose we start from the level $N$, and the maximal number of
companies during the time interval $[0, T]$ is larger than or equal to
$2L$, where $L > N$ is some large number. Then it takes time less than
$T$ to get from the level $L$ to the level $2L$; this will require at
least $L-1$ double jumps upward, for instance, one from $L$ to $L+2$,
another from $L+1$ to $L+3$, etc. Note that such double jumps upward
may ``overlap,'' when there are three or more consecutive jumps upward.

The crucial part in the proof of Proposition~\ref{tails} is to show
that the probability of a double jump upward from $N-1$ to $N+1$ is
small for large $N$. This is done in Lemmas \ref{lemma:double-jump-1},
\ref{lemma:estimation-of-A} and \ref{lemma:double-jump-2}.

Indeed, as we shall see in Lemma~\ref{lemma:after-an-upward-jump},
immediately after the first upward jump, from $N-1$ to $N$, the top
market weight will be less than or equal to $1 - \delta_0$,
where
%
\begin{equation}
\label{delta0} \delta_0:= 1 - (1 - \delta) (1 - \eps_0)
> \delta.
\end{equation}
In other words, the process $\mu(\cdot)$ of market weights will ``stay
away'' from the threshold $1 - \delta$, which it must hit before the
exponential clock $\CE(\lambda_{N})$ rings, for a double jump at
level $N$
to happen. But from Lemma~\ref{lemma} in the \hyperref
[app]{Appendix}, the probability
of this event is at most
%
\begin{eqnarray}
\label{piN} p_N:= 2 \biggl(\frac{( 1 - \delta_0) \vee(1/2)}{1 - \delta}
\biggr)^{\lambda
_N^{1/2}/ \ol{\si}}= 2\exp \bigl(-\al_0\lambda_N^{1/2}
\bigr)
\nonumber
\\[-8pt]
\\[-8pt]
\eqntext{ \mbox{where } \displaystyle\al_0:= \frac{1}{ \ol{\si} } \log \biggl(
\frac{1 - \delta}{( 1 -
\delta_0) \vee(1/2)} \biggr).}
\end{eqnarray}

Then we fix the number of steps $u$ and claim that it is unlikely for
the process to perform 
$L-1$ double jumps upward within $u$ steps. But if the process gets to
the level $2L$ in $M \ge u$ steps, then there will be a lot of jumps
downward, at independent exponential random times. Since there will be
a lot of these random times, we can apply the large deviation theory
and argue that their sum is very likely to be greater than $T$.\vadjust{\goodbreak}

\subsection{Preliminary remarks and estimates}
\label{Prelims}

\subsubsection{Leaving a given level}

The process $\mu(\cdot) = \mathfrak{z}  (X(\cdot) )$ of market
weights evolves in the following manner: As long as it stays on the
$(N-1)$-dimensional manifold $\De^{N, \delta}_+$ (i.e., ``on the level
$N$''), the process $\mu(\cdot)$ evolves as an $N$-dimensional diffusion
governed by the system (\ref{dyna}) of SDEs.

\textit{How does the process $\mu(\cdot)$ leave the level $N$}?
There are
two possibilities:
\begin{longlist}[(ii)]
\item[(i)] An exponential clock $\eta^N_j \backsim\CE(\lambda_N)$ rings; by
construction, the random variable $\eta^N_j$ is independent of the
diffusion given by the system of SDEs in (\ref{dyna}). Then we choose
randomly two companies to merge, \textit{excluding the top one}. As
mentioned in Remark~\ref{rmk:never-suppressed}, this requirement is
essential for ensuring than the merger will not be suppressed [i.e.,
with this proviso we never find ourselves in case B(iv) of
Section~\ref{formcon}].

\item[(ii)] The market weight of one of these $N$ companies, say of the $i$th
one, hits at some time $\tau$ the level $ 1 - \delta$; thus $ \mu_i
(\tau)= 1 - \delta$ and $\sum_{j \neq i} \mu_j (\tau)=\delta$.
Then we
pick a random variable $\xi\backsim F$, independent of the past, and
split the $i$th company into two new companies: these are assigned
market weights $\xi\mu_i(\tau)$ and $(1 - \xi)\mu_{i}(\tau)$.
\end{longlist}

Since $1/2 \le\xi\le1 - \eps_0$ from Assumption~\ref{ass3}, each
of the
resulting two new market weights is at most $(1 - \delta)(1 - \eps_0)
= 1
- \delta_0$ as in (\ref{delta0}); whereas all the other companies,
those unaffected by the split, have market weights bounded from above
by $\delta$. Because $ \eps_0 \in(0, 1/2)$ and $\delta\in(0,
1/6)$, we
have $ \delta< (1 - \delta)/2 < (1 - \delta)(1 - \eps_0)= 1 -
\delta_0 < 1 - \delta
$, so again \textit{all} of the new market weights are bounded away from
the threshold $1 - \delta$.

Let us state this observation in the form of a separate lemma.

\begin{lemma}
For the time $ \taub$ of any upward jump in the process $ \mathcal
{N} (\cdot) $ and with $ \delta_0 $ as in \eqref{delta0}, we have
\[
\mu(\taub+) \in\De^{N^\star, \delta_0}_+\qquad \mbox{for some integer } N^\star
\ge2.
\]
In other words, immediately after any upward jump, the market weight
process is in $\CM^{\delta_0}$ and $ \mu_{(1)}(\taub+ ) \le1 -
\delta_0
$ holds.
\label{lemma:after-an-upward-jump}
\end{lemma}

\subsubsection{Jumping upward, rather than downward}

Let us obtain an upper bound for the conditional probability $\MP
_x
( \tau_m \le\eta^{N_{m-1}}_m | N_{m-1} = N  )$ that the $m$th
jump will be upward rather than downward, given that during the
time-interval $ (T_{m-1}, T_m ) $ the market weight process $ \mu
(\cdot) = \mathfrak{z}  (X(\cdot) ) $ is at a given level $ N
\in\mathbb{N} $.

On the event $ \{ N_{m-1} =N\} $ and during the time-interval $
(T_{m-1}, T_m ) $ with the ``event-time'' $ T_m = T_{m-1} + ( \tau_m
\wedge\eta^{N }_m) $, the process of log-capitalizations evolves as
a system of competing Brownian particles with drifts $g_k = g_{N k}$
and variances $\si_k^2 = \si_{N k}^2 $, for $ k = 1, \ldots, N $.

First, we consider $m = 1$; by the comparison lemma from the \hyperref
[app]{Appendix},
we get
\[
\MP_x \bigl(\tau_1 \le\eta_1^{N_0}
\mid N_0 = N \bigr) = \MP _x \bigl(\tau _1
\le\eta_2^{N_0} \bigr) \le \biggl(\frac{\mu_{(1)}(0)\vee(1/2)
}{1 -
\delta}
\biggr)^{\lambda_{N }^{1/2} / \ol{\si}} ,
\]
because in the notation of Assumption~\ref{ass2}, we have $\ol{\sigma
} \ge
\widetilde{\si}:= \max_{1 \le k \le N}(\si_{Nk})$. Now, back to the
case of general $m$, the strong Markov property gives
%
\begin{eqnarray}
\label{estimate} \MP_x \bigl(\tau_{m+1} \le
\eta^{N_{m }}_{m+1} | N_{m } = N \bigr)& =&
\ME_x \bigl( \MP_{X(T_{m })} \bigl(\tau_1 \le
\eta^{N}_1 \bigr) \bigr)
\nonumber
\\[-8pt]
\\[-8pt]
\nonumber
&\le& 2 \ME_x \biggl(
\frac{\mu_{(1)}(T_{m }+)\vee(1/2)}{1 - \delta
} \biggr)^{\lambda
_{N }^{1/2} / \ol{\si}}.
\end{eqnarray}

\subsubsection{A couple of auxiliary estimates}

In this subsection, we estimate the probability of a double jump from
$N-1$ to $N+1$, and then show that this estimate is in some sense
independent of the past: we can condition on an event which occurred
before the first of these tandem jumps was completed (the jump from
$N-1$ to $N$). Although this conditioning might influence the exact
probability of the double jump, the estimate which is deduced in this
subsection remains unchanged.

We recall the definition of $A(m, N)$ from \eqref{eq:double-jump}.

\begin{lemma}
\label{lemma:double-jump-1}
Fix $ m \ge1, N \ge3$. Then, with $p_N$ as in \eqref{piN}, we have
the following estimate for the probability of a double upward jump:
\begin{eqnarray*}
\MP_x \bigl(A(m, N)\mid N_{m-1} = N-1\bigr) &\equiv&
\MP_x(N_{m+1} = N+1\mid N_m = N,
N_{m-1} = N-1 ) \\
&\le& p_N.
\end{eqnarray*}
\end{lemma}

\begin{pf} If $N_{m-1} = N-1$ and $N_m = N$, then by Lemma~\ref
{lemma:after-an-upward-jump} we have $ \mathfrak{z}  (X(T_m+)
) \in\Delta_+^{N, \delta_0} $, thus $\mu_{(1)}(T_m+) \le1 - \delta
_0$. It
follows from \eqref{estimate} that $ \MP_x (N_{m+1} = N+1\mid N_m =
N, N_{m-1} = N-1) $ satisfies
\[
\MP_x \bigl(\tau_{m+1} \le\eta^{N_m}_{m+1}
| N_m = N, N_{m-1} = N-1 \bigr) \le2 \biggl(
\frac{(1 - \delta_0) \vee(1/2)}{1 -
\delta
} \biggr)^{\lambda_{N }^{1/2} / \ol{\si}} = p_N
\]
in the notation of (\ref{piN}).
\end{pf}

Let us show that the same estimate holds under conditioning on any
event before the first jump. This will be necessary when we consider
the probability of many double jumps. The events $A(m, N)$ and $A(m_1,
N_1)$ are clearly \textit{not independent}, but the estimate still holds.

%
\begin{lemma}
\label{lemma:estimation-of-A}
For every $ m \ge1, N \ge3$ and $A \in\CF(T_m)$, we have
\begin{eqnarray*}
&&\MP_x (N_{m+1} = N+1 \mid N_m = N,
N_{m-1} = N-1, A ) \\
&&\qquad=\MP _x \bigl(A(m, N) \mid
N_m = N , N_{m-1} = N-1, A \bigr) \le p_N.
\end{eqnarray*}
\end{lemma}
\begin{pf}
We claim that for $x_m \in(0, \infty)^N $ with $ \mathfrak{z} (x_m)
\in\De^{N, \delta_0}_+$, we have
\[
\MP_x \bigl(N_{m+1} = N+1 | X(T_m+) =
x_m , N_m = N, N_{m-1} = N-1, A \bigr) \le
p_N.
\]
This follows from the estimate of the previous lemma, and from the fact
that $T_{m+1}$ is a function of the initial condition $x_m = X(T_m+)$
and of the Brownian increments driving the system $Y^{N_m, x_m,
m+1}(\cdot)= Y^{N, x_m, m+1}(\cdot)$; by construction, these increments
are independent of $\CF(T_m)$.

The strong Markov property now completes the proof.
\end{pf}

\begin{cor}
\label{lemma:double-jump-2}
For fixed integers $m_1 < m_2 < \cdots< m_j < m$ and $ N_1,\break   N_2,
\ldots, N_m, N$, we have
\[
\MP_x \bigl(A(m, N)\mid A(m_1, N_1),
A(m_2, N_2), \ldots, A(m_j,
N_j) \bigr) \le p_N.
\]
\end{cor}
\begin{pf}
This is an immediate corollary of Lemma~\ref{lemma:estimation-of-A}:
\begin{eqnarray*}
&&\MP_x \bigl(A(m, N)\mid A(m_1, N_1),
A(m_2, N_2), \ldots, A(m_j,
N_j) \bigr) \\
 &&\qquad = \MP_x \bigl(N_{m+1} = N+1,
N_m = N, N_{m-1} = N-1\mid\\
&&\qquad\quad A(m_1,
N_1), A(m_2, N_2), \ldots,
A(m_j, N_j) \bigr)
\\
& &\qquad\le \MP_x \bigl(N_{m+1} = N+1\mid N_m = N,
N_{m-1} = N-1,\\
&&\qquad\quad A(m_1, N_1),
A(m_2, N_2), \ldots, A(m_j,
N_j) \bigr).
\end{eqnarray*}
But the event $A = A(m_1, N_1)\cap\cdots\cap A(m_j, N_j)$ belongs to
$\CF(T_m)$, because each event $A(m_i, N_i)$, $i=1, \ldots, j$ depends
on the state of the system at stopping times $T_{m_i - 1} , T_{m_i}
, T_{m_i +1}$, 
and all of these are dominated by $T_m $. An application of Lemma~\ref
{lemma:estimation-of-A} completes the proof.
\end{pf}

We have the following consequence of Corollary~\ref{lemma:double-jump-2}.

\begin{cor}
\label{cor:multiple-A}
For fixed $m_1 < \cdots< m_j$ and $N_1, \ldots, N_j$, we have the estimate
\[
\MP_x \bigl(A(m_1, N_1)\cap\cdots\cap
A(m_j, N_j) \bigr) \le p_{N_1}\cdots
p_{N_j}.
\]
\end{cor}

\subsection{The proof of Proposition \texorpdfstring{\protect\ref{tails}}{4.1}}
\label{proof_Prop}

With $N = \mathfrak N(x)$, let us estimate the probability that the
market weight process $ \mu(\cdot) = \mathfrak{z}  (X(\cdot
)
) $ rises during the time-interval $(0,T)$ from the level $N$ to the
level $2L$, where $L > N$. First, the process has to reach the level
$L $; it will get there for the first time as the result of a split,
and immediately after the jump it will be in $\De^{L, \delta_0}_+$. Then
it will have time less than $T$ to reach the level $2L$. For each $n =
2, 3, \ldots,$ the random variable
\[
\Theta_n:= \inf \bigl\{ t \ge0: \mathcal{N} (t) = n \bigr\}
\]
will denote the first time when the counting process $ \mathcal{N}
(\cdot) $ of our model hits the $n$th level (i.e., the first time
$n$ companies are extant).

Suppose we are able to establish the following estimate: \textit{For
every $\be> 0$, there exist $L_0 > N$ and $c_0 > 0$ such that for
every $L > L_0$ and $y \in\De^{L, \delta_0}_+$ we have}
%
\begin{equation}
\label{crucial_estimate} \MP_y (\Theta_{2L} \le T ) \le
c_0e^{-\be L}.
\end{equation}
Then the rest of the proof will follow. Indeed, to get from $x$ to the
level $2L$ in time less than or equal to $T$, the process needs first
to get to the level $L$ by an upward jump. By Lemma~\ref
{lemma:after-an-upward-jump}, immediately after this jump, the process
will be at some point $y \in\De^{L, \delta_0}_+$. Starting from this
point, it has to reach the level $2L$ during the remaining time (which
is of course smaller than $T$). Therefore, integrating over $y \in\De
^{L, \delta_0}_+$ with respect to the distribution of $\mu
(\Theta_{L}
+ )$ and using the strong Markov property, we will get then
\[
\MP_x \Bigl( \max_{0 \le t \le T}\CN(t) \ge2 L \Bigr) =
\MP_x (\Theta_{2L} \le T ) \le c_0e^{-\be L}
,
\]
and the proof will be complete.

Thus, let us try to estimate the $\MP_y$-probability of the event $\{
\Theta_{2L} \le T\} $ in (\ref{crucial_estimate}), for $y \in\De^{L,
\delta_0}_+$. Suppose that it takes the process of market weights $M$
jumps to reach the level $2L $; then for every real number $u>0$ we have
%
\begin{equation}
\label{crucial_estimate2} \MP_y(\Theta_{2L} \le T) =
\MP_y (\Theta_{2L} \le T, M > 3u )+ \MP_y (
\Theta_{2L} \le T, M \le3u ).
\end{equation}
We shall try to find a number $u > L$ such that the event $\{ \Theta
_{2L} \le T, M \le3u\} $ is unlikely, and the event
$\{ \Theta_{2L} \le T, M > 3u\}$ is also unlikely. Let us introduce a
couple of new pieces of notation:
\[
\underline{\lambda}_L:= \min(\lambda_{L+1}, \ldots,
\lambda _{2L-1}),\qquad \overline{\lambda}_L:= \max(
\lambda_3, \ldots, \lambda_L).
\]

First, we estimate the probability $\MP_y (\Theta_{2L} \le T, M >
3u )$ on the right-hand side of (\ref{crucial_estimate2}). To get
from the level $L$ to the level $2L$ in $M$ jumps, one needs to make
$L$ more upward than downward jumps. But the total number of these
jumps is $M$, so the number of downward jumps is $(M-L)/2$. Therefore,
on the event $\{\Theta_{2L} \le T, M > 3u\}$, there are
\[
\frac{M - L}2 > \frac{3u - u}2 = u
\]
downward jumps. If a jump proceeds from the level $i$ to the level
$i-1$, it takes time $\eta_i \backsim\CE(\lambda_i)$ to make this jump
happen (counting from the last one). All these exponential jump times
are independent, so there exist i.i.d. $\CE(1)$ random variables
$\widetilde{\eta}_1, \ldots, \widetilde{\eta}_u$ such that
$\widetilde
{\eta}_i = \lambda_i\eta_i$. We can write
\[
\{\Theta_{2L} \le T, M > 3u\} \subseteq\{\eta_1 +
\cdots+ \eta_u \le T\} = \biggl\{\frac{\widetilde{\eta}_1}{\lambda_1} + \cdots+
\frac
{\widetilde{\eta
}_u}{\lambda_u} \le T \biggr\}.
\]
Since $\lambda_i \le\ol{\lambda}_{2L-1} , i = 3, \ldots, 2L-1$, we have
\[
\biggl\{\frac{\widetilde{\eta}_1}{\lambda_1} + \cdots+ \frac
{\widetilde{\eta
}_u}{\lambda_u} \le T \biggr\} \subseteq
\biggl\{\frac{\widetilde
{\eta}_1 +
\cdots+ \widetilde{\eta}_{u}}{\ol{\lambda}_{2L-1}} \le T \biggr\} = \biggl\{ \frac{\widetilde{\eta}_1 + \cdots+ \widetilde{\eta}_{u}}u \le
\frac
{T\ol{\lambda}_{2L-1}}u \biggr\}.
\]
We apply now techniques from the theory of large deviations, as in the
book \cite{DZBook}, Theorem~2.2.3 and Exercise 2.2.23(c), page 35. The
rate function $\CH$ for the exponential distribution $\CE(1)$ is given
by $\CH(s) = s - 1 - \log s$, for $s > 0$ (this function is denoted by
$\La^*$ in \cite{DZBook}, Section~2.2). For $F = [0, s]$, according
to the remark (c) on page 27 of the book \cite{DZBook} (immediately after
the statement of Theorem~2.2.3), we have
\[
\MP_y \biggl\{\frac{\widetilde{\eta}_1 + \cdots+ \widetilde{\eta}_{u}}u \le s \biggr\} \le 2 \exp \Bigl(-u
\inf_{v \in F}\CH (v) \Bigr).
\]
It is checked that the function $\CH$ is decreasing on $(0, 1]$;
therefore, for $s \in(0, 1)$, we have
\[
\inf_{v \in F}\CH(v) = \inf_{v \in[0, s]}\CH(v) =
\CH(s).
\]
Assuming that $u$ is large enough, namely $ u > L \vee (T \ol
{\lambda
}_{2L-1} )$, we obtain
\[
\MP_y \biggl\{\frac{\widetilde{\eta}_1 + \cdots+ \widetilde{\eta}_{u}}u \le \frac{T\ol{\lambda}_{2L-1}}u \biggr\}
\le2 \exp \biggl(-u \CH \biggl(\frac{T\ol{\lambda}_{2L-1}}u \biggr) \biggr),
\]
therefore,
%
\begin{eqnarray}
\label{est11} \MP_y (\Theta_{2L} \le T, M > 3u ) &\le&
\MP_y \biggl\{\frac
{\tilde{\eta}_1 + \cdots+ \tilde{\eta}_u}u \le\frac{T\ol
{\lambda
}_{2L-1}}u \biggr\}
\nonumber
\\[-8pt]
\\[-8pt]
\nonumber
& \le&2
\exp \biggl(-u \CH \biggl(\frac{T\ol{\lambda
}_{2L-1}}u \biggr) \biggr) =:
\Sigma_1 (u).
\end{eqnarray}

Now, let us estimate the probability $\MP_y (\Theta_{2L} \le T, M
\le3u )$ on the right-hand side of (\ref{crucial_estimate2}). In
order to reach the level $2L$ starting from $L$ in no more than $3u$
jumps, we need to have at least $L-1$ double jumps upward, as discussed
in the preliminary remarks of Section~\ref{sec: overview}.

One of these double jumps is from level $L$ to level $L+2$, occurring
at step $m_1$. Another is from level $L+1$ to level $L+3$, occurring at
step $m_2$, etc., up to a double jump upward from level $2L - 2$ to
level $2L$, occurring at step $m_{L-1}$. The subset $\{m_1, \ldots,
m_{L-1}\} \subseteq\{1, \ldots, 3u-1\}$ with $1 \le m_1 < m_2 <
\cdots
< m_{L-1} \le3u - 1$ can be chosen in
\[
\pmatrix{3u-1
\cr
L-1} \le\pmatrix{3u
\cr
L} \le\frac{(3u)^L}{L!}
\]
different ways. For a given subset $\{m_1, \ldots, m_{L-1}\} \subseteq
\{1, \ldots, 3u-1\}$, Corollary~\ref{cor:multiple-A} states that the
probability $\MP_y (\Theta_{2L} \le T, M \le3u )$ is no more than
\begin{eqnarray*}
&&\MP_y \bigl(A(m_1, L+1), A(m_2, L+2),
\ldots, A(m_{L-1}, 2L-1) \bigr)\\
&&\qquad \le p_{L+1}\cdots
p_{2L-1}
\\
& &\qquad\le 2^{L-1}\exp \bigl(-\al_0 \bigl(\lambda_{L+1}^{1/2}
+ \cdots+ \lambda_{2L-1}^{1/2} \bigr) \bigr)
\le2^{L-1}\exp \bigl(-\al _0(L-1)\underline{
\lambda}_L^{1/2} \bigr),
\end{eqnarray*}
%
and thus
%
\begin{equation}
\label{est2} \MP_y (M \le3u, \Theta_{2L} \le T ) \le
\frac
{(3u)^L}{L!}2^{L-1}\exp \bigl(-\al_0(L-1)\underline{
\lambda }_L^{1/2} \bigr) =: \Sigma_2 (u).
\end{equation}

It follows from the estimates in (\ref{est11}) and (\ref{est2}) that
the probability of the event $ \{ \Theta_{2L} \le T \} $ which we
would like to estimate, as in (\ref{crucial_estimate}) and (\ref
{crucial_estimate2}), is\vspace*{2pt}
\begin{eqnarray*}
\MP_y (\Theta_{2L} \le T ) &=& \MP_y (
\Theta_{2L} \le T, M \le3u ) + \MP_y (
\Theta_{2L} \le T, M > 3u ) \\[2pt]
&\le& \Sigma_1(u) +
\Sigma_2(u).
\end{eqnarray*}
Here, $u > L$. Note that $\lambda_L \asymp cL^{\al}$ as $L \to\infty
$, so
$\ol{\lambda}_{2L-1} \asymp2^{\al}cL^{\al}$ and $\underline
{\lambda}_L \asymp
c2^{\al}L^{\al}$ as $L \to\infty$.

We need now to fix the undetermined parameter $u $: we shall let $u =
\lceil kL^{\al\vee1}\rceil$ for large enough $k > 0$. The function
$\CH
(\cdot)$ satisfies $ \CH(s) \ge-(1/2)\log s $ for $s \in(0, s_0)$
for some constant $s_0 \in(0, 1)$. Therefore, for large enough $L$ and
$k$ we have the following estimates for these two summands. First, let
us estimate $\Sigma_1 (u)$ of~(\ref{est11}): we have
\[
\CH \biggl(\frac{T\ol{\lambda}_{2L-1}}{u} \biggr) \ge -\frac{1}2\log \biggl(
\frac{T\ol{\lambda}_{2L-1}}{u} \biggr) \ge \frac{1}2\log\frac
{k}{T2^{\al
}c} =:
k_0,
\]
therefore,
\[
\Sigma_1(u) = 2\exp \biggl(-u \CH \biggl(\frac{T\ol{\lambda
}_{2L-1}}u \biggr)
\biggr) \le 2\exp \bigl(-kk_0L^{\al\vee1} \bigr).
\]
By taking $k$ large enough, we can make $kk_0$ as large as we want.
Since $\al\vee1 \ge1$, this proves that $\Sigma_1(u)$ decreases
faster than any exponential function as $u \to\infty$. This completes
the proof for this summand $\Sigma_1(u)$. The other summand
\[
\Sigma_2(u) = \frac{(3u)^L}{L!} 2^{L-1}\exp \bigl(-
\al_0(L-1) \underline{\lambda}_L^{1/2} \bigr)
\]
from (\ref{est2}) decreases faster than any $e^{-\be L}$ as $L \to
\infty$ for any fixed $\be> 0$, because $\underline{\lambda}_L
\asymp
cL^{\al}$ and
\[
\log\Sigma_2(u) = L\log(3u) - \log(L!) + (L-1)\log2 - \al
_0(L-1) \underline{\lambda}_L^{1/2}
\]
is asymptotically smaller as $L \to\infty$ than
\[
(\al\vee1)L\log L - L\log L - \al_0 c (L-1)L^{(1/2)\vee(\al/2)} \asymp -
\al_0 c L^{1 + (\al\vee1)/2}.
\]
This establishes the bound of (\ref{crucial_estimate}), so the proof of
the proposition is complete. 

\subsection{The proof of Theorem \texorpdfstring{\protect\ref{thmm2}}{2.2}}
\label{proof_thmm2}

The general philosophy of the proof is as follows: For any given real
number $ T >0 $ we shall try to find an \textit{equivalent martingale
measure}, that is, a probability measure $\MQ_T$ on $\CF(T)$ with the
following properties:
\begin{longlist}[(ii)]
\item[(i)] $ \MQ_T \sim\MP$ on $\CF(T) $;

\item[(ii)] for every portfolio $\pi(\cdot) $, the wealth process $ V^{\pi}
( t), 0 \le t \le T $ is a $\MQ_T$-martingale.
\end{longlist}

Suppose this is done; take two portfolios $\pi(\cdot)$ and $\rho
(\cdot
)$, and assume for a moment $\pi(\cdot)$ allows an arbitrage
opportunity relative to $\rho(\cdot)$ on a given time horizon $[0,T]$
with $T \in( 0, \infty)$. Then the conditions of \eqref{arbitr} hold
with respect to the measure $\MP$ and, therefore, with respect to the
measure $\MQ_T$ as well. But
\[
V^{\pi}(t) - V^{\rho}(t),\qquad 0 \le t \le T
\]
is a $\MQ_T$-martingale with initial value zero, therefore, $ \ME
^{\MQ
_T} (V^{\pi}(T) - V^{\rho}(T) ) = 0$ holds in contradiction to
\eqref{arbitr}. This contradiction completes the proof, that arbitrage
is not possible.

For a model with a given, fixed number of companies, an equivalent
martingale measure is constructed thus: a \textsc{Girsanov} change of
measure ensures that each capitalization process is a martingale with
respect to the new measure, and the wealth process is a stochastic
integral with these processes as integrators. As a result, the wealth
process is also a martingale with respect to the new measure. But here
the number of extant companies fluctuates, so we shall carry out a
\textsc{Girsanov} construction up to the first jump, then carry out the
same construction with the new number of stocks up to the second jump,
and so on. We do this in a number of steps, as follows.

\textit{Step} 1: First, as a preliminary step, let us consider the
CBP-based market model from Section~\ref{FNC} with the dynamics of
\eqref{fixed} and under an appropriate filtration $ \mathbb{G} = \{
\CG
(t)\}_{ 0 \le t < \infty} $ that satisfies the usual conditions.

Consider the processes $\Upsilon_i (\cdot) = \{\Upsilon_i(t), 0 \le
t < \infty\}, i = 1, \ldots, N$, given by
%
\begin{eqnarray}
\label{Ups} \Upsilon_i(t):= \int_0^t
\frac{\mathrm{d}X_i(s)}{X_i(s)} = \SL_{k=1}^N \int_0^t
\mathbf{ 1}_{\{ X_i(t) = X_{(k)}(t)\}} \bigl( g_{k} \,\mathrm{d}t +
\si_{k} \,\mathrm{d}W_i(t) \bigr),
\nonumber
\\[-8pt]
\\[-8pt]
\eqntext{ 0 \le t < \infty.}
\end{eqnarray}
Over each interval $ [ 0, T] $ with $T>0$ a given real number, each
$ \Upsilon_i(\cdot\wedge T) $ can be turned into a martingale by
the change of probability measure
\[
\MQ_T (A) = \MP \bigl(Z(T) \mathbf{ 1}_A \bigr),\qquad A \in
\CG(T) .
\]
Here,
\[
Z(t) = \exp \bigl(- M (t) - \tfrac{1}2\langle M\rangle(t) \bigr),\qquad 0 \le
t < \infty,
\]
and the $\MP$-martingale $M (\cdot) = \{ M(t), 0 \le t < \infty\} $
is given by
\[
M(t) = \SL_{k=1}^N \biggl(\frac{g_k}{\si_k} \biggr)
\SL_{i=1}^N \int_0^t
\mathbf{1}_{\{X_i(u) = X_{(k)}(u)\}}\,\mathrm{d}W_i(u) ,\qquad 0 \le t < \infty.
\]
The quadratic variation of this martingale $M (\cdot)$ is
%
\begin{equation}
\label{estim} \langle M\rangle(t) = \SL_{k=1}^N \biggl(
\frac{g_k}{\si_k} \biggr)^2 \SL_{i=1}^N \int
_0^t \mathbf{1}_{\{X_i(u) = X_{(k)}(u)\}} \,\mathrm{d}u \le
t N\max_{1 \le k \le N} \biggl(\frac
{g_k}{\si_k} \biggr)^2.
\end{equation}
Using the \textsc{Novikov} condition (\cite{KS1991}, Proposition~3.5.12),
we see that $Z(\cdot)$ is a $\MP$-martingale, and, therefore,
$\mathbf{
Q}_T$ a probability measure on $ \mathcal{F} (T) $. Also, the
quadratic variations of the processes $\Upsilon_i(\cdot)$ from (\ref
{Ups}) are given by
%
\begin{equation}
\label{quadraticvariation}\qquad \langle\Upsilon_i \rangle(t) = \SL_{k=1}^N
\si^2_{k} \int_0^t
\mathbf{ 1}_{\{ X_i(u) = X_{(k)}(u)\}} \,\mathrm{d}u \le t\max_{1 \le k \le N}
\si_k^2 ,\qquad i=1, \ldots, N ,
\end{equation}
whereas the independence of the Brownian motions $W_i(\cdot)$ and
$W_j(\cdot)$ gives that
%
\begin{equation}
\label{crossvariation} \langle\Upsilon_i, \Upsilon_j \rangle(
\cdot) \equiv0 \qquad\mbox{holds for } 1 \le i \neq j \le N.
\end{equation}

It is clear from this discussion that $ X_i (\cdot\wedge T) = \exp
 ( Y_i (\cdot\wedge T) - (1/2) \langle Y_i \rangle(\cdot\wedge T)
 ) , i=1, \ldots, N $ are martingales (with zero
cross-variations) under the probability measure $ \mathbf{ Q}_T $,
which thus earns the appellation of Equivalent Martingale Measure (EMM)
for the model of Section~\ref{FNC}.

\begin{rmk}
\label{NoDiv}
It follows now easily, that the CBP-based market model of Section~\ref
{FNC} is not diverse. For if this model were diverse on some
time-horizon $[0,T]$ of finite length, Proposition~6.2 of \cite{FK2009}
would proscribe for it EMMs, such as the probability measure $\mathbf{
Q}_T$ just constructed. See also an illuminating discussion in~\cite{BFK2005},
Section~7.
\end{rmk}

\textit{Step} 2:
Now, let $M^{N, x, n} (\cdot) = \{M^{N, x, n}(t), 0 \le t < \infty\}
$ be the same martingale $M(\cdot)$ for the copy of a CBP-based market model
\[
\bigl(\exp \bigl(Y^{N, x, n}_1(\cdot) \bigr), \ldots, \exp
\bigl(Y^{N, x,
n}_N(\cdot) \bigr) \bigr)^\prime
\]
from Section~\ref{formcon}. This model has parameters $g_k = g_{Nk},
\si_k =
\si_{Nk}, k = 1, \ldots, N$, the initial condition is $x$, and all the
processes $ \{ M^{N, x, n} (\cdot)\}_{n \in\mathbb{N}} $ are
independent. Also, denote by
\[
\Upsilon^{N, x, n}_i(\cdot), \qquad i = 1, \ldots, N,
\]
the processes $\Upsilon_i (\cdot) $ of (\ref{Ups}) for this copy of the
market model. Slightly abusing notation, we define
%
\begin{equation}
\label{M} M(t) = \SL_{m \in\mathbb{N}_0} M^{N_m, x_m, m+1}(t\wedge
T_{m+1} - t\wedge T_m), \qquad 0 \le t < \infty
\end{equation}
in the notation of Section~\ref{formcon}. This is an $\mathbb{F}=\{\CF
(t)\}_{
t \ge0}$-local martingale, with localizing sequence $\{T_m\}_{m \in
\mathbb{N}_0}$ and quadratic variation
\[
\langle M\rangle(t) = \SL_{m \in\mathbb{N}_0} \bigl\langle M^{N_m, x_m, m+1} \bigr
\rangle( t \wedge T_{m+1} - t \wedge T_{m}).
\]

\textit{Step} 3:
Let us verify the \textsc{Novikov} condition
%
\begin{equation}
\label{novikov} \ME_x \bigl[ \exp \bigl(\tfrac{1}2 \langle
M \rangle(T) \bigr) \bigr] < \infty,\qquad 0 \le T < \infty
\end{equation}
of \cite{KS1991}, Proposition~3.5.12. The expression \eqref{estim}
leads to the estimate
\[
\langle M\rangle(t) \le\SL_{m \in\mathbb{N}_0} (t \wedge T_{m+1} - t
\wedge T_{m} ) N_m\max_{1 \le k \le N_m} \biggl(
\frac
{g_{N_m k }}{\si_{N_m k}} \biggr)^2.
\]
Since Assumption~\ref{ass2} implies that
\[
\frac{|g_{N_m k}|}{\si_{N_mk} } \le\frac{\ol{g}}{\underline{\si
}} =: C < \infty \qquad\mbox{holds for all } m
\ge0 , k \ge1 ,
\]
we get
%
\begin{equation}
\label{<M>} \langle M\rangle(T) \le C^2T\max_{0 \le t \le T}
\mathcal {N}(t) ,
\end{equation}
and so the left-hand side in \eqref{novikov} can be estimated as
\[
\ME_x \biggl( \exp \biggl[\frac{C^2T}{2} \cdot\max
_{0 \le t \le
T}\CN(t) \biggr] \biggr).
\]
But this quantity is finite for any given real number $T \in( 0,
\infty
)$ because of \eqref{subexp} from Proposition~\ref{tails}, establishing
the Novikov condition \eqref{novikov}. Thus, the stochastic exponential
%
\begin{equation}
\label{Z} Z (\cdot) = \bigl\{Z(t), 0 \le t <\infty\bigr\},\qquad Z(t) = \exp
\bigl[-M(t) - \tfrac{1}2 \langle M\rangle(t) \bigr]
\end{equation}
is a $\MP$-martingale, and we can define a new probability measure
$\MQ
_T$ that satisfies
%
\begin{equation}
\label{eq: cue} \mathrm{d}\MQ_T = Z(T) \,\mathrm{d}\MP \qquad\mbox{on }
\CF(T), \mbox{ for each given } T \in[ 0, \infty).
\end{equation}

\textit{Step} 4:
We shall show now that, for any given real number $ T>0 $, the
wealth process $V^{\pi}(\cdot\wedge T)$ is a martingale for any given
portfolio $\pi(\cdot)$, under the new measure $ \MQ_T $ just
constructed. [This is very clearly the case for the cash portfolio, as
$ V^\kappa(\cdot) \equiv1$.]

We write the equation for $V^{\pi}(\cdot)$ in the form
\begin{eqnarray}
\int_0^t \frac{\mathrm{d}V^{\pi}(u)}{V^{\pi}(u)} =
\SL_{m \in
\mathbb{N}_0} \SL_{i=1}^{N_m} \int_0^{ t \wedge T_{m+1} - t \wedge
T_m}
\pi_i( T_m+u) \,\mathrm{d}\Upsilon^{N_m, x_m, m+1}_i(u)
,\nonumber\\
\eqntext{0 \le t < \infty.}
\end{eqnarray}
Here, the processes $\Upsilon_i^{N_m, x_m, m+1}(\cdot\wedge T)$ are
defined via
\[
\Upsilon^{N_m, x_m, m+1}_i(t):= \int_0^{ t\wedge T_{m+1} - t\wedge
T_m}
\frac{ \mathrm{d}X^{N_m, x_m, m+1}_i(s) }{X^{N_m, x_m,
m+1}_i(s)} ,\qquad i = 1, \ldots, N_m ,
\]
and are local martingales under the probability measure $\MQ_T $ in
(\ref{eq: cue}). 
Therefore, the process $L^\pi(\cdot\wedge T) $, defined via
\[
L^{\pi} (t):= \int_0^t
\frac{\mathrm{d}V^{\pi}(u)}{V^{\pi}(u)} = \SL_{m \in\mathbb{N}_0} \SL_{i=1}^{N_m} \int
_0^{ t \wedge T_{m+1}
- t
\wedge T_m} \pi_i( T_m+u)
\,\mathrm{d}\Upsilon^{N_m, x_m, m+1}_i(u) ,
\]
is an $(\mathbb{F}, \MQ_T)$-local martingale. The value process
$V^{\pi
}(\cdot)$ is the stochastic exponent of $L^{\pi}(\cdot)$, namely
\[
V^{\pi}(\cdot) = \exp \bigl(L^{\pi}(\cdot) -
\tfrac{1}2 \bigl\langle L^{\pi} \bigr\rangle(\cdot) \bigr).
\]
If we can establish the \textsc{Novikov} condition
%
\begin{equation}
\label{Noiv} \mathbf E^{\MQ_T} \bigl[ \exp \bigl(\tfrac{1}2
\bigl\langle L^{\pi} \bigr\rangle(T) \bigr) \bigr]< \infty,
\end{equation}
then it will turn out that $V^{\pi} ( \cdot\wedge T)$ is an $(\mathbb
{F}, \MQ_T)$-martingale, as indeed we set out to show at the start of
this proof.

Indeed, from \eqref{quadraticvariation} and \eqref{crossvariation} we
see that the processes $\Upsilon_i^{N_m, x_m, m+1}(\cdot) , i=1,
\ldots, N_m $ have zero cross-variations, and quadratic variations
\begin{eqnarray*}
\bigl\langle\Upsilon_i^{N_m, x_m, m+1} \bigr\rangle(t )& = &\SL
_{k=1}^{N_m} \si^2_{k} \int
_0^{ t\wedge T_{m+1} - t\wedge T_m} \mathbf{ 1}_{  \{ X_i^{N_m, x_m, m+1} (u) = X_{(k)}^{N_m, x_m,
m+1} (u) \}} \,\mathrm{d}u
\\
&\le& t \max_{1 \le k \le
N_m}\si^2_{N_mk}.
\end{eqnarray*}
Therefore, the process $ \Upsilon_\pi^{N_m, x_m, m+1}(\cdot) =
\{
\Upsilon_{\pi}^{N_m, x_m, m+1}(t), 0 \le t < \infty \} $ given by
\[
\Upsilon_{\pi}^{N_m, x_m, m+1}(t):= \SL_{i=1}^{N_m}
\int_0^t\pi_i( T_m +u)
\,\mathrm{d}\Upsilon_i^{N_m, x_m, m+1}(u)
\]
has quadratic variation $
 \langle\Upsilon_{\pi}^{N_m, x_m, m+1} \rangle(T) \le
TK_\pi^2 \cdot\max_{1 \le k \le N_m}\si_{N_m k}^2 $,
because $ |\pi_i(t)| \le K_\pi< \infty$ holds for all $ 0 \le t <
\infty, i = 1, \ldots, N_m $. This gives
\[
\bigl\langle L^{\pi} \bigr\rangle(T) = \sum
_{k \in\mathbb{N}_0} \bigl\langle\Upsilon_{\pi}^{N_k, x_k, k+1}
\bigr\rangle ({T\wedge T_{k+1} - T\wedge T_k} ) \le
TK_\pi^2 \cdot\mathop{\max_{N \ge2}}_{ 1 \le k \le N}
\si_{Nk}^2 \le T K_{\pi}^2 \ol{
\si}^{ 2} ,
\]
and  property (\ref{Noiv}) is proved.

\subsection{Some open questions}
\label{open}

(I) The above proof used the boundedness of the portfolio $\pi(\cdot)
$ in a crucial way. It would be very interesting to see whether
arbitrage in this (or in a related) model with splits and mergers might
exist with more general, unbounded portfolios.\vspace*{-6pt}
\begin{longlist}[(iii)]
\item[(II)] The estimate (\ref{<M>}) also gives the bound
\[
\langle M\rangle(T)- \langle M\rangle(t) \le C^2 T \cdot\max
_{t \le\theta\le T}\mathcal{N}(\theta)
\]
for every $ t \in[0, T]$. From the theory of Bounded Mean Oscillation
(BMO) Martingales as developed, for instance, in the book \cite
{BMOBook}, in order to show that the stochastic exponential $Z(\cdot)$
of (\ref{Z}), (\ref{M}) is a martingale, it suffices to show that
the~process $\mathbf{ E}  ( \langle M \rangle(T) - \langle M \rangle
(t) | \mathcal{F} (t)  ) , 0 \le t \le T $ is uniformly
bounded. Thus, on the strength of the last display, it is enough to
show that the process
\[
\mathbf{ E} \Bigl( \max_{ t \le\theta\le T} \mathcal{N}(\theta)\Big |
\mathcal{F} (t) \Bigr) ,\qquad 0 \le t \le T
\]
is uniformly bounded. If this can be done, it might obviate the need to
establish the sub-exponential bound of Proposition~\ref{tails}.

\item[(III)] It would be very interesting to decide whether absence of
arbitrage, say with respect to the market portfolio, survives when one
begins to constrain the splits and/or mergers that can happen over a
given period of time, or along genealogies of companies produced by any
given split (one could mandate, e.g., that the resulting
companies cannot be touched for a certain amount of time). We believe
not, but this issue remains to be settled.

\item[(IV)] It would be interesting to extend the above generic analysis, by
allowing for some ``idiosyncratic'' features in the model. These can
take the form of allowing the growth rates and the local co-variation
rates for the different assets to depend, not only on the rank, but
also on the name of the particular company (e.g., as in Ichiba et al.
\cite{Ichiba11} or the so-called \textit{second-order model} from
\cite
{FIK2013b}). They could also take the form of giving strategic control
to the various companies, on decisions such as whether to engage in a
merger or not. Our present model does not allow such features.
\end{longlist}
\begin{appendix}\label{app}
\section*{Appendix: A comparison lemma}
%
\begin{lemmaa}
\label{lemma}
Consider a CBP-based market model as described in Definition~\ref{CBP}
of Section~\ref{FNC}. Assume in the manner of \eqref{condition} that
$
g_1 \le\min_{2 \le k \le N}g_k $, and let
\[
\tau:= \inf \bigl\{t \ge0: \exists i = 1, \ldots, N , \mbox{s.t. }
\mu_i(t) = 1 - \delta \bigr\} ,\qquad \widetilde{\si}: = \max
_{1 \le k \le N}\si_k.
\]
Then for an independent random variable $ \eta$, exponentially
distributed with parameter $\lambda> 0$, we have
%
\begin{equation}
\label{est} \MP(\tau\le\eta) \le2 \biggl(\frac{\mu_{(1)}(0)\vee{1}/2}{1 -
\delta
}
\biggr)^{\widetilde{\si}^{-1}\lambda^{1/2}}.
\end{equation}
\end{lemmaa}

The idea behind the argument of the proof is as follows. We can rewrite
the stopping time $\tau$ as
$
\tau= \inf \{t \ge0: \mu_{(1)}(t) = 1 - \delta \} $;
indeed, whenever a market weight reaches $ 1 - \delta$, then it also
gets to be the largest market weight, because $1 - \delta> 1/2$. The
process $\log\mu_{(1)} (\cdot)$ ``reflects off'' $\log\mu_{(2)}
(\cdot)
$, as made precise in~\eqref{localtimes}; it behaves like an It\^o
process, but when it collides with $\log\mu_{(2)} (\cdot) $ a positive
local time term emerges as a result of the collision.

Now, we would like to replace $\log\mu_{(1)} (\cdot)$ by something
larger. Consider a similar process $ U (\cdot)$, now reflected at
$\log(1/2)$; the second top market weight never gets above $1/2$, and
for this new reflection pattern the logarithm of the top market weight
will reflect earlier, so the resulting reflected process $ U (\cdot)
$ will be larger.
In particular, we shall have $ \tau\ge\widetilde{\tau}:= \inf
\{t \ge0: U(t) = 1 - \delta \} $;
and the probability that the exponential clock
(which is responsible for mergers) rings later than $ \tau$ (the
time of a split), will be smaller than the probability that this
exponential clock rings later than $ \widetilde{\tau} $. But $U(\cdot
)$ is reflected Brownian motion, so we can calculate this latter
probability explicitly.

\begin{pf*}{Proof of Lemma \ref{lemma}}
Let us derive an equation for the dynamics of $ \log
\mu
_{(1)}(\cdot) $. We recall the expression for the dynamics of $ \log
\mu_i(\cdot) $ from (\ref{dyna}), and denote by $\Lambda_{(k, \ell)}
(\cdot)= \{\Lambda_{(k, \ell)}(t), t \ge0\}$ the local time at the
origin of the continuous semimartingale
\[
\log\mu_{(k)} (\cdot) - \log\mu_{(\ell)} (\cdot)=
Y_{(k)} (\cdot)- Y_{(\ell)} (\cdot) \qquad \mbox{for } 1 \le k < \ell
\le N.
\]
We have $\Lambda_{(k, \ell)} (\cdot) \equiv0$ if $\ell- k \ge2$, see
\cite{Ichiba11}, Lemma~1, as well as
%
\begin{equation}
\label{logmu1} \mathrm{d}\log\mu_{(1)}(t) = \SL_{i=1}^N
\mathbf{ 1}_{\{\mu_i(t)
= \mu
_{(1)}(t)\}} \,\mathrm{d}\log\mu_{i}(t) +
\frac{1}{ 2 } \,\mathrm {d}\Lambda_{(1,2)}(t)
\end{equation}
from \cite{BG2008}. We also note from \cite{Ichiba11} that the set $\{t
\ge0\mid\mu_{(k)}(t) = \mu_{(1)}(t)\} = \{t \ge0\mid Y_{(k)}(t) =
Y_{(1)}(t)\}$ has zero Lebesgue measure, for $k = 2, \ldots, N$.
Introduce the following notation:
\begin{eqnarray*}
\be(t)&: =& g_1 - \SL_{k=1}^Ng_k
\mu_{(k)}(t) - \frac{1}2\SL_{k=1}^N\si
_k^2 \bigl(\mu_{(k)}(t) -
\mu_{(k)}^2(t) \bigr),\\ 
a(t)&:=&
\si_1^2 \bigl(1 - \mu_{(1)}(t)
\bigr)^2 + \SL_{k=2}^N\si_k^2
\mu_{(k)}^2(t) >0.
\end{eqnarray*}
Using \eqref{dyna}, we can rewrite \eqref{logmu1} in the notation of
(\ref{Levy}) [denoting by $V (\cdot)= \{ V(t), t \ge0\}$ yet another
one-dimensional standard $ \{ \CF(t) \}_{t \ge0}$-Brownian motion]
as
\begin{eqnarray*}
\mathrm{d}\log\mu_{(1)}(t) &=& \be(t) \,\mathrm{d}t + \si_1\,
\mathrm {d}B_1(t) - \SL_{k=1}^N
\si_k\mu_{(k)}(t)\,\mathrm{d}B_k(t) +
\frac{1}{
2 } \,\mathrm{d}\Lambda_{(1,2)}(t)
\\
&= &\be(t)\,\mathrm{d}t + \sqrt{a(t)} \,\mathrm{d}V(t)+ \frac{1}{ 2 } \,\mathrm{d}
\Lambda_{(1,2)}(t).
\end{eqnarray*}
For the coefficient $a(t)$, we get the following estimate:
$a(t) \le\si_1^2 + \break  \max_{2 \le k \le N}\si_k^2 \le2 \widetilde
{\si}^2$.
Also, the coefficient $a(t)$ is bounded away from zero, at least until
the moment $\tau$, because for $t \le\tau$, $\mu_{(1)}(t) \le1 -
\delta
$, and
$a(t) \ge\si_1^2\delta^2$. Since $\mu_{(k)}(t) \in[0, 1]$, we get:
$\mu
_{(k)}(t) - \mu^2_{(k)}(t) \ge0$. It is easy to get the following
estimate for $\be(t)$:
\begin{eqnarray*}
\be(t)& \le& g_1 \bigl(1 - \mu_{(1)}(t) \bigr) -
\SL_{k=2}^Ng_k\mu _{(k)}(t) \le
g_1 \bigl(1 - \mu_{(1)}(t) \bigr) - \min
_{2 \le k \le N}g_k\cdot\SL _{k=2}^N
\mu_{(k)}(t)\\ & =& \Bigl(g_1 - \min_{2 \le
k \le N}g_k
\Bigr) \bigl(1 - \mu_{(1)}(t) \bigr) \le0.
\end{eqnarray*}
Let us make a random time change, using Lemma~2 from \cite{MyOwn1} (for
$\si= 1$ in the notation of this lemma). The time change is as follows:
\[
t = T(s) = \inf\bigl\{t \ge0\mid\De(t) \ge s\bigr\},\qquad s = \De(t):= \int
_0^ta(v)\,\mathrm{d}v, t \in[0, \tau].
\]
Denoting by $\ol{V} (\cdot) =  \{ \ol{V}(s), s \ge0 \}$ yet
another standard Brownian motion, and
\begin{eqnarray*}
Z (\cdot) &= &\bigl\{ Z(s), s \ge0\bigr\} ,\qquad Z(s) = \log\mu_{(1)}
\bigl(T(s)\bigr) ,
\\
\ol{\Lambda} (\cdot) &=& \bigl\{\ol{\Lambda}(s), s \ge0 \bigr\} ,\qquad \ol {\Lambda}(s)
= \tfrac{1}{ 2 } \Lambda_{(1, 2)}\bigl(T(s)\bigr) ,\\
 \ga(s) &=&
\frac{\be(T(s))}{a(T(s))}
\end{eqnarray*}
we get the following equation:
\[
\mathrm{d}Z(s) = \ga(s) \,\mathrm{d}s + \mathrm{d}\ol{V}(s) + \mathrm {d}\ol{
\Lambda}(s)\qquad \mbox{for } s \le\De(\tau).
\]
One important remark: $\ga(s) \le0$ for all $s \le\De(\tau)$.
Let us establish one useful property of $\ol{\La}$:
%
\begin{equation}
\label{eq:special-comp} \mbox{if } Z(s) = \log\mu_{(1)}\bigl(T(s)\bigr) \ge
\log(1/2) \qquad\mbox{then } \mathrm{d}\ol{\La}(s) = 0.
\end{equation}
Indeed, if $\mu_{(1)}(t) > \mu_{(2)}(t)$, then $\mathrm{d}\La(t) = 0$;
in words, $\La$ is constant in a neighborhood of $t$. Therefore, if
$\mu
_{(1)}(T(s)) > \mu_{(2)}(T(s))$, then $\mathrm{d}\ol{\La}(s) \equiv
\,\mathrm{d}\La(T(s)) = 0$. In particular, if $\mu_{(1)}(T(s)) > 1/2$,
then $\mu_{(1)}(T(s)) > 1/2 \ge\mu_{(2)}(T(s))$, and \mbox{$\mathrm{d}\ol
{\La
}(s) = 0$}.

We shall show the comparison $Z(\cdot) \le Z_0(\cdot)$, where $Z_0
(\cdot)= \{Z_0(s), s \ge0\}$ is a one-dimensional Brownian motion
with zero drift and unit dispersion, starting from $\log\mu_{(1)}(0)$
and reflected at $\log(1/2)$, namely
\[
\mathrm{d}Z_0(s) = \mathrm{d}\ol{V}(s) + \mathrm{d}
\Lambda^0(s).
\]
Here, $\Lambda^0 (\cdot)= \{ \Lambda^0(s), s \ge0\}$ is the local
time of this reflecting Brownian motion at the site $\log(1/2)$. More
precisely, let us show that
%
\begin{equation}
\label{eq:comp-Z} Z(s) \le Z_0(s), \qquad s \le\De(\tau).
\end{equation}

The proof of \eqref{eq:comp-Z} proceeds along the same lines as in
\cite{IWBook},
Chapter~6, Theorem~1.1, but with some adjustments which are
necessary because of the local time terms. We define $\psi(x):= x_+^3$
for $x \in\BR$; then $\psi\in C^2(\BR)$, and for $s \le\De(\tau)$
%
\begin{eqnarray}
\label{eq:final-comp} \psi\bigl(Z(s) - Z_0(s)\bigr) &=& \int
_0^s\psi'\bigl(Z(u) -
Z_0(u)\bigr)\ga(u)\mathrm {d}u
\nonumber
\\[-8pt]
\\[-8pt]
\nonumber
&&{} + \int_0^s
\psi'\bigl(Z(u) - Z_0(u)\bigr) \bigl(\mathrm{d}\ol{
\Lambda}(u) - \mathrm {d}\Lambda^0(u)\bigr).
\end{eqnarray}
This does not contain stochastic integrals, because in the expression for
$Z(s) - Z_0(s)$ they cancel out. Let us show that the right-hand side
of \eqref{eq:final-comp} is nonpositive. Indeed, $\ga(s) \le0$, and
$\psi'(x) = 3x_+^2 \ge0$, so the first integral in the right-hand side
of \eqref{eq:final-comp} is nonpositive. Also, when $Z(s) > Z_0(s)$,
we have: $Z(s) > Z_0(s) \ge\log(1/2)$. Therefore, for these $s \le
\De
(\tau)$, from \eqref{eq:special-comp} we get: $\mathrm{d}\ol
{\Lambda
}(s) = 0$,
and so $ (\mathrm{d}\ol{\Lambda}(u) - \mathrm{d}\Lambda
^0(u) )
\le0$. Once again using the fact that $\psi'(x) = 3x_+^2 \ge0$, we
get that the second integral in \eqref{eq:final-comp} is also
nonpositive. So the whole expression $\psi(Z(s) - Z_0(s)) \le0$, and
this implies $Z(s) \le Z_0(s)$, $s \le\De(\tau)$.

Let $\tau_0:= \inf\{t \ge0\mid Z_0(t) = \log(1 - \delta)\}$. Since
$\tau
$ is the hitting time by $\mu_{(1)}$ of $1 - \delta$, or, in other words,
by $\log\mu_{(1)}$ of $\log(1 - \delta)$, we get that $\Delta(\tau
)$ is
the hitting time by $Z(\cdot) = \log\mu_{(1)}(T(\cdot))$ of the level
$\log(1 - \delta)$. Therefore, by comparison~\eqref{eq:comp-Z}, we have
$\tau_0 \le\Delta(\tau)$. But $\De'(t) = a(t) \le2 \widetilde
{\si
}^2 $, so
\[
\De(t) = \int_0^ta(s)\,\mathrm{d}s \le 2
\widetilde{\si}^2t ,\qquad t \le \tau.
\]
In particular, $\De(\tau) \le2\widetilde{\si}^2\tau$, and $\tau
\ge
\tau_0/(2 \widetilde{\si}^2)$. Therefore,
\[
\MP(\tau\le\eta) \le \MP\bigl(\tau_0 \le2 \widetilde{
\si}^2 \eta \bigr).
\]
Using \cite{BSBook}, Part II, Section~3, formula 1.1.2, and the fact
that $ 2 \widetilde{\si}^2 \eta$ is exponentially distributed with
parameter $
\lambda/(2\widetilde{\si}^2) $, we get
\[
\MP \bigl( \tau_0 \le2 \widetilde{\si}^2 \eta \bigr) =
\MP _{\widetilde
{x}} \Bigl( \sup_{0 \le s \le2 \widetilde{\si}^2 \eta
}\bigl|\ol{B}(s)\bigr| \ge y \Bigr)
= \frac{\operatorname{ch} (\widetilde
{x}\sqrt{\lambda
} \widetilde{\si}^{-1} ) }{\operatorname{ch} (\widetilde
{y}\sqrt{\lambda}
\widetilde{\si}^{-1} )}
\]
with $\widetilde{y} = \log(1 - \delta) - \log(1/2), \widetilde{x}=
(\log(\mu_{(1)}(0)) - \log(1/2) )^+$. From the elementary inequality
$
(e^z /2) \le\operatorname{ch}z \equiv(e^z + e^{-z})/2 \le e^z, z \ge0
$
we conclude
\[
\frac{\operatorname{ch} (\widetilde{x}\sqrt{\lambda}
\widetilde{\si}^{-1} )}{\operatorname{ch}
 (\widetilde{y}\sqrt{\lambda} \widetilde{\si}^{-1} )} \le 2\exp \bigl(- (\widetilde{y} - \widetilde{x} )\sqrt{
\lambda} \widetilde{\si}^{-1} \bigr) ,
\]
and it is then straightforward to rewrite the right-hand side as \eqref{est}.
This completes the proof.
\end{pf*}
\end{appendix} 

\section*{Acknowledgements}

The authors would like to thank E. Robert \textsc{Fernholz},
Jean-Pierre \textsc{Fouque}, Soumik \textsc{Pal}, Vasileios \textsc
{Papathanakos}, Walter\break \textsc{Schachermayer}, Johannes \textsc{Ruf}
and Philip \textsc{Whitman} for very helpful discussions, and Johannes
\textsc{Ruf} for his detailed comments. They are obliged to Winslow
\textsc{Strong} for
an updated version of the manuscript \cite{Strong2011}. They are deeply
indebted to the referee for reading of the manuscript 
very carefully, and for making many incisive comments and suggestions.



\begin{thebibliography}{34}

\bibitem{FFR2007}
%
\begin{barticle}[auto:parserefs-M02]
\bauthor{\bsnm{Audrino},~\bfnm{Francesco}\binits{F.}},
\bauthor{\bsnm{Fernholz},~\bfnm{E.~Robert}\binits{E.~R.}} \AND
\bauthor{\bsnm{Ferretti},~\bfnm{Roberto~G.}\binits{R.~G.}}
(\byear{2007}).
\btitle{A forecasting model for stock market diversity}.
\bjournal{Ann. Finance}
\bvolume{3}
\bpages{213--240}.
\end{barticle}
%

\bptok{imsref}%
\endbibitem

\bibitem{BFK2005}
%
\begin{barticle}[mr]
\bauthor{\bsnm{Banner},~\bfnm{Adrian~D.}\binits{A.~D.}},
\bauthor{\bsnm{Fernholz},~\bfnm{E.~Robert}\binits{E.~R.}} \AND
\bauthor{\bsnm{Karatzas},~\bfnm{Ioannis}\binits{I.}}
(\byear{2005}).
\btitle{Atlas models of equity markets}.
\bjournal{Ann. Appl. Probab.}
\bvolume{15}
\bpages{2296--2330}.
\bid{doi={10.1214/105051605000000449}, issn={1050-5164}, mr={2187296}}
\end{barticle}
%

\bptok{imsref}%
\endbibitem

\bibitem{BG2008}
%
\begin{barticle}[mr]
\bauthor{\bsnm{Banner},~\bfnm{Adrian~D.}\binits{A.~D.}} \AND
\bauthor{\bsnm{Ghomrasni},~\bfnm{Raouf}\binits{R.}}
(\byear{2008}).
\btitle{Local times of ranked continuous semimartingales}.
\bjournal{Stochastic Process. Appl.}
\bvolume{118}
\bpages{1244--1253}.
\bid{doi={10.1016/j.spa.2007.08.001}, issn={0304-4149}, mr={2428716}}
\end{barticle}
%

\bptok{imsref}%
\endbibitem

\bibitem{Bass1987}
%
\begin{barticle}[mr]
\bauthor{\bsnm{Bass},~\bfnm{R.~F.}\binits{R.~F.}} \AND
\bauthor{\bsnm{Pardoux},~\bfnm{{\'E}.}\binits{{\'E}.}}
(\byear{1987}).
\btitle{Uniqueness for diffusions with piecewise constant coefficients}.
\bjournal{Probab. Theory Related Fields}
\bvolume{76}
\bpages{557--572}.
\bid{doi={10.1007/BF00960074}, issn={0178-8051}, mr={0917679}}
\end{barticle}
%

\bptok{imsref}%
\endbibitem

\bibitem{BSBook}
%
\begin{bbook}[mr]
\bauthor{\bsnm{Borodin},~\bfnm{Andrei~N.}\binits{A.~N.}} \AND
\bauthor{\bsnm{Salminen},~\bfnm{Paavo}\binits{P.}}
(\byear{2002}).
\btitle{Handbook of {B}rownian Motion---Facts and Formulae},
\bedition{2nd} ed.
\bpublisher{Birkh\"auser},
\blocation{Basel}.
\bid{doi={10.1007/978-3-0348-8163-0}, mr={1912205}}
\end{bbook}
%

\bptok{imsref}%
\endbibitem

\bibitem{CP2010}
%
\begin{barticle}[mr]
\bauthor{\bsnm{Chatterjee},~\bfnm{Sourav}\binits{S.}} \AND
\bauthor{\bsnm{Pal},~\bfnm{Soumik}\binits{S.}}
(\byear{2010}).
\btitle{A phase transition behavior for {B}rownian motions interacting
through their ranks}.
\bjournal{Probab. Theory Related Fields}
\bvolume{147}
\bpages{123--159}.
\bid{doi={10.1007/s00440-009-0203-0}, issn={0178-8051}, mr={2594349}}
\end{barticle}
%

\bptok{imsref}%
\endbibitem

\bibitem{DZBook}
%
\begin{bbook}[mr]
\bauthor{\bsnm{Dembo},~\bfnm{Amir}\binits{A.}} \AND
\bauthor{\bsnm{Zeitouni},~\bfnm{Ofer}\binits{O.}}
(\byear{2010}).
\btitle{Large Deviations Techniques and Applications}.
\bseries{Stochastic Modelling and Applied Probability}
\bvolume{38}.
\bpublisher{Springer},
\blocation{Berlin}.
\bid{doi={10.1007/978-3-642-03311-7}, mr={2571413}}
\end{bbook}
%

\bptok{imsref}%
\endbibitem

\bibitem{F2002}
%
\begin{bbook}[mr]
\bauthor{\bsnm{Fernholz},~\bfnm{E.~Robert}\binits{E.~R.}}
(\byear{2002}).
\btitle{Stochastic Portfolio Theory}.
\bseries{Applications of Mathematics (New York)}
\bvolume{48}.
\bpublisher{Springer},
\blocation{New York}.
\bid{doi={10.1007/978-1-4757-3699-1}, mr={1894767}}
\end{bbook}
%

\bptok{imsref}%
\endbibitem

\bibitem{FIK2013}
%
\begin{barticle}[mr]
\bauthor{\bsnm{Fernholz},~\bfnm{E.~Robert}\binits{E.~R.}},
\bauthor{\bsnm{Ichiba},~\bfnm{Tomoyuki}\binits{T.}} \AND
\bauthor{\bsnm{Karatzas},~\bfnm{Ioannis}\binits{I.}}
(\byear{2013}).
\btitle{Two {B}rownian particles with rank-based characteristics and
skew-elastic collisions}.
\bjournal{Stochastic Process. Appl.}
\bvolume{123}
\bpages{2999--3026}.
\bid{doi={10.1016/j.spa.2013.03.019}, issn={0304-4149}, mr={3062434}}
\end{barticle}
%

\bptok{imsref}%
\endbibitem

\bibitem{FIK2013b}
%
\begin{barticle}[mr]
\bauthor{\bsnm{Fernholz},~\bfnm{E.~Robert}\binits{E.~R.}},
\bauthor{\bsnm{Ichiba},~\bfnm{Tomoyuki}\binits{T.}} \AND
\bauthor{\bsnm{Karatzas},~\bfnm{Ioannis}\binits{I.}}
(\byear{2013}).
\btitle{A second-order stock market model}.
\bjournal{Ann. Finance}
\bvolume{9}
\bpages{439--454}.
\bid{doi={10.1007/s10436-012-0193-2}, issn={1614-2446}, mr={3082660}}
\end{barticle}
%

\bptok{imsref}%
\endbibitem

\bibitem{FK2005}
%
\begin{barticle}[auto:parserefs-M02]
\bauthor{\bsnm{Fernholz},~\bfnm{E.~Robert}\binits{E.~R.}} \AND
\bauthor{\bsnm{Karatzas},~\bfnm{Ioannis}\binits{I.}}
(\byear{2005}).
\btitle{Relative arbitrage in volatility-stabilized markets}.
\bjournal{Ann. Finance}
\bvolume{1}
\bpages{149--177}.
\end{barticle}
%

\bptok{imsref}%
\endbibitem

\bibitem{FK2009}
%
\begin{barticle}[auto:parserefs-M02]
\bauthor{\bsnm{Fernholz},~\bfnm{E.~Robert}\binits{E.~R.}} \AND
\bauthor{\bsnm{Karatzas},~\bfnm{Ioannis}\binits{I.}}
(\byear{2009}).
\btitle{Stochastic portfolio {Theory}: An {Overview}}.
\bjournal{Handb. Numer. Anal.}
\bvolume{15}
\bpages{89--167}.
\end{barticle}
%

\bptok{imsref}%
\endbibitem

\bibitem{FKK2005}
%
\begin{barticle}[mr]
\bauthor{\bsnm{Fernholz},~\bfnm{E.~Robert}\binits{E.~R.}},
\bauthor{\bsnm{Karatzas},~\bfnm{Ioannis}\binits{I.}} \AND
\bauthor{\bsnm{Kardaras},~\bfnm{Constantinos}\binits{C.}}
(\byear{2005}).
\btitle{Diversity and relative arbitrage in equity markets}.
\bjournal{Finance Stoch.}
\bvolume{9}
\bpages{1--27}.
\bid{doi={10.1007/s00780-004-0129-4}, issn={0949-2984}, mr={2210925}}
\end{barticle}
%

\bptok{imsref}%
\endbibitem

\bibitem{IchibaThesis}
%
\begin{bmisc}[auto]
\bauthor{\bsnm{Ichiba},~\bfnm{Tomoyuki}\binits{T.}}
(\byear{2009}).
\bhowpublished{Topics in multi-dimensional diffusion theory:
{A}ttainability, reflection, ergodicity and rankings.
Ph.D. thesis, Columbia University,
ProQuest LLC, Ann Arbor, MI.}
\bid{mr={2717746}}
\end{bmisc}
%

\bptok{imsref}%
\endbibitem

\bibitem{IK2010}
%
\begin{barticle}[mr]
\bauthor{\bsnm{Ichiba},~\bfnm{Tomoyuki}\binits{T.}} \AND
\bauthor{\bsnm{Karatzas},~\bfnm{Ioannis}\binits{I.}}
(\byear{2010}).
\btitle{On collisions of {B}rownian particles}.
\bjournal{Ann. Appl. Probab.}
\bvolume{20}
\bpages{951--977}.
\bid{doi={10.1214/09-AAP641}, issn={1050-5164}, mr={2680554}}
\end{barticle}
%

\bptok{imsref}%
\endbibitem

\bibitem{IKS2013}
%
\begin{barticle}[mr]
\bauthor{\bsnm{Ichiba},~\bfnm{Tomoyuki}\binits{T.}},
\bauthor{\bsnm{Karatzas},~\bfnm{Ioannis}\binits{I.}} \AND
\bauthor{\bsnm{Shkolnikov},~\bfnm{Mykhaylo}\binits{M.}}
(\byear{2013}).
\btitle{Strong solutions of stochastic equations with rank-based coefficients}.
\bjournal{Probab. Theory Related Fields}
\bvolume{156}
\bpages{229--248}.
\bid{doi={10.1007/s00440-012-0426-3}, issn={0178-8051}, mr={3055258}}
\end{barticle}
%

\bptok{imsref}%
\endbibitem

\bibitem{IPS2012}
%
\begin{barticle}[mr]
\bauthor{\bsnm{Ichiba},~\bfnm{Tomoyuki}\binits{T.}},
\bauthor{\bsnm{Pal},~\bfnm{Soumik}\binits{S.}} \AND
\bauthor{\bsnm{Shkolnikov},~\bfnm{Mykhaylo}\binits{M.}}
(\byear{2013}).
\btitle{Convergence rates for rank-based models with applications to
portfolio theory}.
\bjournal{Probab. Theory Related Fields}
\bvolume{156}
\bpages{415--448}.
\bid{doi={10.1007/s00440-012-0432-5}, issn={0178-8051}, mr={3055264}}
\bptnote{check volume, check pages, check year}%
\end{barticle}
%

\bptok{imsref}%
\endbibitem

\bibitem{Ichiba11}
%
\begin{barticle}[mr]
\bauthor{\bsnm{Ichiba},~\bfnm{Tomoyuki}\binits{T.}},
\bauthor{\bsnm{Papathanakos},~\bfnm{Vassilios}\binits{V.}},
\bauthor{\bsnm{Banner},~\bfnm{Adrian}\binits{A.}},
\bauthor{\bsnm{Karatzas},~\bfnm{Ioannis}\binits{I.}} \AND
\bauthor{\bsnm{Fernholz},~\bfnm{E.~Robert}\binits{E.~R.}}
(\byear{2011}).
\btitle{Hybrid atlas models}.
\bjournal{Ann. Appl. Probab.}
\bvolume{21}
\bpages{609--644}.
\bid{doi={10.1214/10-AAP706}, issn={1050-5164}, mr={2807968}}
\end{barticle}
%

\bptok{imsref}%
\endbibitem

\bibitem{IWBook}
%
\begin{bbook}[mr]
\bauthor{\bsnm{Ikeda},~\bfnm{Nobuyuki}\binits{N.}} \AND
\bauthor{\bsnm{Watanabe},~\bfnm{Shinzo}\binits{S.}}
(\byear{1989}).
\btitle{Stochastic Differential Equations and Diffusion Processes},
\bedition{2nd} ed.
\bseries{North-Holland Mathematical Library}
\bvolume{24}.
\bpublisher{North-Holland},
\blocation{Amsterdam}.
\bid{mr={1011252}}
\end{bbook}
%

\bptok{imsref}%
\endbibitem

\bibitem{KPS2012}
%
\begin{barticle}[auto:parserefs-M02]
\bauthor{\bsnm{Karatzas},~\bfnm{Ioannis}\binits{I.}},
\bauthor{\bsnm{Pal},~\bfnm{Soumik}\binits{S.}} \AND
\bauthor{\bsnm{Shkolnikov},~\bfnm{Mykhaylo}\binits{M.}}
(\byear{2016}).
\btitle{Systems of Brownian particles with asymmetric collisions}.
\bjournal{Ann. Inst. Henri Poincar\'{e} Probab. Stat.}
\bvolume{52}
\bpages{323--354}.
\end{barticle}
%

\bptok{imsref}%
\endbibitem

\bibitem{KS1991}
%
\begin{bbook}[mr]
\bauthor{\bsnm{Karatzas},~\bfnm{Ioannis}\binits{I.}} \AND
\bauthor{\bsnm{Shreve},~\bfnm{Steven~E.}\binits{S.~E.}}
(\byear{1991}).
\btitle{Brownian Motion and Stochastic Calculus},
\bedition{2nd} ed.
\bseries{Graduate Texts in Mathematics}
\bvolume{113}.
\bpublisher{Springer},
\blocation{New York}.
\bid{doi={10.1007/978-1-4612-0949-2}, mr={1121940}}
\end{bbook}
%

\bptok{imsref}%
\endbibitem

\bibitem{K2008}
%
\begin{barticle}[auto:parserefs-M02]
\bauthor{\bsnm{Kardaras},~\bfnm{Constantinos}\binits{C.}}
(\byear{2008}).
\btitle{Balance, growth and diversity of financial markets}.
\bjournal{Ann. Finance}
\bvolume{4}
\bpages{369--397}.
\end{barticle}
%

\bptok{imsref}%
\endbibitem

\bibitem{BMOBook}
%
\begin{bbook}[mr]
\bauthor{\bsnm{Kazamaki},~\bfnm{Norihiko}\binits{N.}}
(\byear{1994}).
\btitle{Continuous Exponential Martingales and {BMO}}.
\bseries{Lecture Notes in Math.}
\bvolume{1579}.
\bpublisher{Springer},
\blocation{Berlin}.
\bid{mr={1299529}}
\end{bbook}
%

\bptok{imsref}%
\endbibitem

\bibitem{OR2006}
%
\begin{barticle}[auto:parserefs-M02]
\bauthor{\bsnm{Osterrieder},~\bfnm{Jorg~R.}\binits{J.~R.}} \AND
\bauthor{\bsnm{Rheinlander},~\bfnm{Thorsten}\binits{T.}}
(\byear{2006}).
\btitle{Arbitrage opportunities in diverse markets via a
non-equivalent measure change}.
\bjournal{Ann. Finance}
\bvolume{2}
\bpages{287--301}.
\end{barticle}
%

\bptok{imsref}%
\endbibitem

\bibitem{PP2008}
%
\begin{barticle}[mr]
\bauthor{\bsnm{Pal},~\bfnm{Soumik}\binits{S.}} \AND
\bauthor{\bsnm{Pitman},~\bfnm{Jim}\binits{J.}}
(\byear{2008}).
\btitle{One-dimensional {B}rownian particle systems with
rank-dependent drifts}.
\bjournal{Ann. Appl. Probab.}
\bvolume{18}
\bpages{2179--2207}.
\bid{doi={10.1214/08-AAP516}, issn={1050-5164}, mr={2473654}}
\end{barticle}
%

\bptok{imsref}%
\endbibitem

\bibitem{PS2010}
%
\begin{barticle}[mr]
\bauthor{\bsnm{Pal},~\bfnm{Soumik}\binits{S.}} \AND
\bauthor{\bsnm{Shkolnikov},~\bfnm{Mykhaylo}\binits{M.}}
(\byear{2014}).
\btitle{Concentration of measure for {B}rownian particle systems
interacting through their ranks}.
\bjournal{Ann. Appl. Probab.}
\bvolume{24}
\bpages{1482--1508}.
\bid{doi={10.1214/13-AAP954}, issn={1050-5164}, mr={3211002}}
\end{barticle}
%

\bptok{imsref}%
\endbibitem

\bibitem{Ruf2013}
%
\begin{bmisc}[auto:parserefs-M02]
\bauthor{\bsnm{Ruf},~\bfnm{Johannes}\binits{J.}} \AND
\bauthor{\bsnm{Runggaldier},~\bfnm{Wolfgang}\binits{W.}}
(\byear{2013}).
\bhowpublished{%
A systematic approach to constructing market models with arbitrage.
Preprint. Available at \arxivurl{arXiv:1309.1988}}.
\end{bmisc}
%

\bptok{imsref}%
\endbibitem

\bibitem{MyOwn1}
%
\begin{barticle}[mr]
\bauthor{\bsnm{Sarantsev},~\bfnm{Andrey}\binits{A.}}
(\byear{2014}).
\btitle{On a class of diverse market models}.
\bjournal{Ann. Finance}
\bvolume{10}
\bpages{291--314}.
\bid{doi={10.1007/s10436-013-0245-2}, issn={1614-2446}, mr={3193785}}
\end{barticle}
%

\bptok{imsref}%
\endbibitem

\bibitem{MyOwn6}
%
\begin{bmisc}[auto:parserefs-M02]
\bauthor{\bsnm{Sarantsev},~\bfnm{Andrey}\binits{A.}}
(\byear{2015}).
\bhowpublished{%
Infinite systems of competing Brownian particles.
Preprint. Available at \arxivurl{arXiv:1403.4229}}.
\end{bmisc}
%

\bptok{imsref}%
\endbibitem

\bibitem{MyOwn5}
%
\begin{bmisc}[auto:parserefs-M02]
\bauthor{\bsnm{Sarantsev},~\bfnm{Andrey}\binits{A.}}
(\byear{2015}).
\bhowpublished{%
Multiple collisions in systems of competing Brownian particles.
Preprint. Available at \arxivurl{arXiv:1309.2621}}.
\end{bmisc}
%

\bptok{imsref}%
\endbibitem

\bibitem{MyOwn3}
%
\begin{barticle}[mr]
\bauthor{\bsnm{Sarantsev},~\bfnm{Andrey}\binits{A.}}
(\byear{2015}).
\btitle{Triple and simultaneous collisions of competing {B}rownian particles}.
\bjournal{Electron. J. Probab.}
\bvolume{20}
\bpages{1--28}.
\bid{doi={10.1214/EJP.v20-3279}, issn={1083-6489}, mr={3325099}}
\bptnote{check pages}%
\end{barticle}
%

\bptok{imsref}%
\endbibitem

\bibitem{S2011}
%
\begin{barticle}[mr]
\bauthor{\bsnm{Shkolnikov},~\bfnm{Mykhaylo}\binits{M.}}
(\byear{2011}).
\btitle{Competing particle systems evolving by interacting L\'evy processes}.
\bjournal{Ann. Appl. Probab.}
\bvolume{21}
\bpages{1911--1932}.
\bid{doi={10.1214/10-AAP743}, issn={1050-5164}, mr={2884054}}
\end{barticle}
%

\bptok{imsref}%
\endbibitem

\bibitem{Strong2011}
%
\begin{barticle}[mr]
\bauthor{\bsnm{Strong},~\bfnm{Winslow}\binits{W.}}
(\byear{2014}).
\btitle{Fundamental theorems of asset pricing for piecewise
semimartingales of stochastic dimension}.
\bjournal{Finance Stoch.}
\bvolume{18}
\bpages{487--514}.
\bid{doi={10.1007/s00780-014-0230-2}, issn={0949-2984}, mr={3232014}}
\end{barticle}
%

\bptok{imsref}%
\endbibitem

\bibitem{FS2011}
%
\begin{barticle}[mr]
\bauthor{\bsnm{Strong},~\bfnm{Winslow}\binits{W.}} \AND
\bauthor{\bsnm{Fouque},~\bfnm{Jean-Pierre}\binits{J.-P.}}
(\byear{2011}).
\btitle{Diversity and arbitrage in a regulatory breakup model}.
\bjournal{Ann. Finance}
\bvolume{7}
\bpages{349--374}.
\bid{doi={10.1007/s10436-010-0175-1}, issn={1614-2446}, mr={2879921}}
\end{barticle}
%

\bptok{imsref}%
\endbibitem
\end{thebibliography}
%
%





\printaddresses
\end{document}